\documentclass[a4paper,11pt,reqno]{amsart}

\usepackage[T1]{fontenc}			
\usepackage[utf8]{inputenc}
\usepackage[italian, english]{babel}
\usepackage{amsmath, amsfonts, amssymb, amsthm, amscd, mathtools}
\usepackage{geometry}
\usepackage[square,numbers]{natbib}	
\usepackage{bbm, bm}
\usepackage{xspace} 
\usepackage{enumerate, enumitem}
\usepackage{hyperref}
\usepackage{xcolor}
\usepackage{soul}
\usepackage{mathrsfs}
\usepackage{dsfont}
\usepackage{comment}
\usepackage{float}

\newcommand{\bbE}{{\ensuremath{\mathbb E}} }
\newcommand{\bbF}{{\ensuremath{\mathbb F}} }

\newcommand{\bbP}{{\ensuremath{\mathbb P}} }

\newcommand{\cC}{{\ensuremath{\mathcal C}} }

\newcommand{\cF}{{\ensuremath{\mathcal F}} }

\newcommand{\cP}{{\ensuremath{\mathcal P}} }
\newcommand{\cQ}{{\ensuremath{\mathcal Q}} }

\newcommand{\cU}{{\ensuremath{\mathcal U}} }

\newcommand{\cZ}{{\ensuremath{\mathcal Z}} }

\newcommand{\bfu}{{\ensuremath{\mathbf u}} }

\newcommand{\dd}{{\ensuremath{\mathrm d}} }
\newcommand{\de}{{\ensuremath{\mathrm e}} }

\newcommand{\dC}{{\ensuremath{\mathrm C}} }

\newcommand{\dL}{{\ensuremath{\mathrm L}} }

\newcommand{\rE}{\mathscr{E}}

\newcommand{\rL}{\mathscr{L}}

\newcommand{\R}{\mathbb{R}}

\newcommand{\N}{\mathbb{N}}

\newcommand{\ind}{\ensuremath{\mathbf{1}}}

\DeclarePairedDelimiter{\abs}{\lvert}{\rvert}
\DeclarePairedDelimiter{\norm}{\lVert}{\rVert}
\DeclarePairedDelimiterX{\inprod}[2]{\langle}{\rangle}{#1, #2}

\renewcommand{\epsilon}{\varepsilon}
\newcommand{\changetheta}
{
\let\temp\theta
\let\theta\vartheta
\let\vartheta\temp
}
\newcommand{\changephi}
{
\let\temp\phi
\let\phi\varphi
\let\varphi\temp
}

\newtheorem{theorem}{Theorem}[section]
\newtheorem{lemma}[theorem]{Lemma}
\newtheorem{proposition}[theorem]{Proposition}
\newtheorem{corollary}[theorem]{Corollary}

\newtheorem{definition}[theorem]{Definition}
\newtheorem{assumption}[theorem]{Assumption}

\newtheorem{remark}[theorem]{Remark}

\newtheoremstyle{bluehead}
  {}
  {}
  {\itshape}
  {}
  {\color{blue}\bfseries}
  {.}
  {5pt plus 1pt minus 1pt}
  {}
\theoremstyle{bluehead}
\newtheorem{theoremblue}[theorem]{Theorem}

\newtheorem{propositionblue}[theorem]{Proposition}

\definecolor{darkviolet}{rgb}{0.58, 0.0, 0.83}

\definecolor{gre}{rgb}{0.03,0.50,0.03}

\newcommand{\blue}[1]{{\color{blue}#1}}

\numberwithin{equation}{section}

\hypersetup{pdftitle={Existence and Uniqueness Results for a Mean-field Game of Optimal Investment},pdfauthor={Alessandro Calvia, Salvatore Federico, Giorgio Ferrari, Fausto Gozzi},hidelinks,%
pdfstartview={XYZ null null 1.4},colorlinks=true,linkcolor=blue,urlcolor=blue}

\newcommand{\mail}[1]{\href{mailto:#1}{\normalfont\texttt{#1}}}

\makeatletter
\def\@setthanks{\vspace{-\baselineskip}\def\thanks##1{\@par##1\@addpunct.}\thankses}
\makeatother

\title[A Mean-field Game of Optimal Investment]{Existence and Uniqueness Results \\ for a Mean-field Game of Optimal Investment}
\author[A.~Calvia]{Alessandro~Calvia\textsuperscript{\MakeLowercase{a},1,\textdagger}}
\thanks{\noindent \textsuperscript{a} Politecnico di Milano, Department of Mathematics,
Piazza Leonardo da Vinci 32, 20133 Milano (Italy).}
\author[S.~Federico]{Salvatore~Federico\textsuperscript{\MakeLowercase{b},2,\textdagger}}
\thanks{\noindent \textsuperscript{b} University of Bologna, Department of Mathematics, Piazza di Porta San Donato 5, 40126 Bologna (Italy).}
\author[G.~Ferrari]{Giorgio~Ferrari\textsuperscript{\MakeLowercase{c},3}}
\thanks{\noindent \textsuperscript{c} Bielefeld University, Center for Mathematical Economics (IMW), Universit\"atsstrasse 25, 33615 Bielefeld (Germany).}
\author[F.~Gozzi]{Fausto~Gozzi\textsuperscript{\MakeLowercase{d},4}}
\thanks{\noindent \textsuperscript{d} LUISS University, Department of Economics and Finance, Viale Romania 32, 00197 Roma (Italy).
\\
\noindent \textsuperscript{\textdagger} Member of the Gruppo Nazionale per l’Analisi Matematica, la Probabilità e le loro Applicazioni (GNAMPA) of the Istituto Nazionale di Alta Matematica "Francesco Severi" (INdAM).
\\
\noindent \textsuperscript{1} E-mail: \mail{alessandro.calvia@polimi.it}.
\\
\noindent \textsuperscript{2} E-mail: \mail{salvatore.federico@unibo.it}.
\\
\noindent \textsuperscript{3} E-mail: \mail{giorgio.ferrari@uni-bielefeld.de}.
\\
\noindent \textsuperscript{4} E-mail: \mail{fgozzi@luiss.it}.
}

\begin{document}

\changephi
\changetheta
\allowdisplaybreaks
	
\begin{abstract}
We establish the existence and uniqueness of the equilibrium for a stochastic \mbox{mean-field} game of optimal investment. The analysis covers both finite and infinite time horizons, and the mean-field interaction of the representative company with a mass of identical and indistinguishable firms is modeled through the time-dependent price at which the produced good is sold. At equilibrium, this price is given in terms of a nonlinear function of the expected (optimally controlled) production capacity of the representative company at each time.
The proof of the existence and uniqueness of the mean-field equilibrium relies on a priori estimates and the study of nonlinear integral equations, but employs different techniques for the finite and infinite horizon cases. Additionally, we investigate the deterministic counterpart of the mean-field game under study.
\end{abstract}
\maketitle


\smallskip

\noindent \textbf{AMS 2020:} 35Q89; 47H10; 49N10; 49N80; 91A07; 91B38; 91B70.

\smallskip



\bigskip

\section{Introduction}
\label{sec:intro}

In this paper, we provide existence, uniqueness, and characterization results for the equilibrium of a mean-field model of optimal investment with competition \`a la Cournot. In the absence of interventions, the production capacity of the representative company evolves stochastically over time as a geometric Brownian motion, and its level can be increased through investment, subject to quadratic costs. The representative company discounts profits and costs at a constant rate and aims to maximize total expected profits from production, net of investment costs. Instantaneous profits depend linearly on the company's production capacity (thus, production occurs at full capacity) and on the time-dependent price of the produced good. The mean-field equilibrium investment and average production processes $(\widehat{\bf{u}}, \widehat{q})$ are such that expected total net profits are maximized and, assuming an isoelastic inverse demand function, the price is given in terms of a nonlinear function of the average optimally controlled production at each time (see also \citet{achdouetal2014:pdemodels} and Section~\ref{sec:eco} below). We are able to prove the existence and uniqueness of the equilibrium pair $(\widehat{ \bf{u}}, \widehat{q})$ when the problem's time horizon is finite or infinite. These equilibria are characterized as the unique solution to a nonstandard system of forward-backward ordinary differential equations. Furthermore, we demonstrate that the existence and uniqueness carry over to the deterministic counterpart of our model, essentially using the same techniques as in the stochastic case.

The investment problem under study falls under the category of mean-field games with scalar interaction. Mean-field games, independently introduced by \citet{LasryLions07}, and Huang, Caines, and Malham\'e \cite{HuangMalhameCaines06}, represent limit models of non-cooperative symmetric $N$-player games with mean-field interaction as the number of players $N$ tends to infinity. An exhaustive review of mean-field models can be found in the two-volume book by~\citet{carmona2018probabilistic}.

In our context, the consistency condition that the equilibrium price aligns at each time with a decreasing nonlinear function of the expected (optimally controlled) production capacity can be viewed as the limit, as the number $N$ of identical and indistinguishable companies operating in the market diverges, of the requirement that price inversely depends on the aggregate production of the entire economy, scaled by a factor of $1/N$. As discussed in \citet{HuangMalhameCaines06}, this scaling can be justified by considering situations where ``an increasing number of firms join together to serve an increasing number of consumers'' (see the discussion in~\citet{HuangMalhameCaines06}, after Equation (2.4) therein).

The equilibrium construction for a given time horizon $T\in (0,+\infty]$ follows a three-step approach. Firstly, given a deterministic path of the average production $q \coloneqq (q_s)_{s \in [0,T]}$, we solve the representative company's optimal investment problem. This is of linear-quadratic type and has an explicit solution, since the production capacity evolves as a geometric Brownian motion and is linearly dependent on the investment process, and because the performance criterion has linear dependence on the production capacity and quadratic dependence on the investment cost. Moreover, the optimal control $\widehat \bfu^{(T,q)}$ is deterministic, for any fixed average production trajectory $q$. In the second step, we calculate the expectation of the optimally controlled production capacity process, which is easily computable given the explicit representation of the state process. Finally, in the third step, we impose the consistency condition, requiring that this expectation must match $q_s$ for each time $s \in [0,T]$. This condition leads to a nonlinear integral equation for the equilibrium average production trajectory $\widehat{q} \coloneqq (\widehat{q}_s)_{s \in [0,T]}$. By construction the control $\widehat \bfu \coloneqq \widehat \bfu^{(T,\widehat{q})}$ and the function $\widehat{q}$ form a mean-field equilibrium. 

It is important to note that although the company's optimal control problem for a fixed average production trajectory $q$ is linear-quadratic, the overall mean-field problem is \emph{not} of linear-quadratic type, due to the nonlinear dependence of the net profit functional on $q$. In this regard, our results differ from those presented in \citet{Bensoussan16, DelarueFT19, tchuendom}, among others, which focus on linear-quadratic mean-field games. {{It is also worth stressing that some of our results -- particularly the existence and uniqueness results -- extend to the case of an instantaneous profit function that is more general than the isoelastic demand function considered in this paper, as well as to more general mean-field interactions, provided that suitable monotonicity conditions are imposed on the problem's data. We refer to Section~\ref{sec:eco} for further details.}}

{{For our existence and uniqueness results, we do not rely on the study of the PDE system associated with the mean-field game (the Hamilton-Jacobi-Bellman equation and the Fokker-Planck equation), nor on the related master equation; we refer to the subsequent literature review for a discussion of related studies offering a PDE perspective on the problem. Instead, we provide an \emph{ad hoc} analysis that hinges on the linear-quadratic optimization problem at hand and on the nonlinear mean-field interaction term under consideration.

This leads to a characterization of the unique equilibrium, both in the cases $T < +\infty$ and $T = +\infty$, through a nonlinear integral equation for the average production $\widehat{q}$. However, different technical arguments are used to prove existence and uniqueness in the finite- and infinite-horizon settings. Indeed, the integral equation that uniquely characterizes the equilibrium average production $\widehat{q}$ does not fall within the standard theory of integro-differential equations (see, e.g.~\citet{gripenberg1990:integraleq, lakshmikantham1995:integraleq}), since we face an initial-value problem with a backward-in-time integral term, which requires a careful specific analysis.

More in detail, the proof of existence relies on a priori estimates and Schauder's fixed-point theorem when $T < +\infty$, whereas in the case $T = +\infty$ it relies on a priori estimates and the Fr\'echet--Kolmogorov compactness theorem in $L^p$. Uniqueness is established in both cases through suitable contradiction arguments that exploit structural properties of mean-field equilibria. Notably, when $T = +\infty$, we also prove that the time-dependent equilibrium $\widehat{q}$ converges monotonically to the stationary average production level, which is obtained as the unique constant solution of the integral equation characterizing $\widehat{q}$.
}}

Given the geometric dynamics of the production capacity, the linear dependence of the profit functional on the controlled state variable, the quadratic costs of investment, and the fact that the mean-field interaction is of scalar type and only involves the expected value of the latter, it is not surprising that the mean-field equilibrium is deterministic and does not depend on the volatility coefficient $\sigma$ appearing in the production capacity's dynamics. Guided by this observation, we also consider the deterministic counterpart of the previously discussed mean-field game and show that a unique equilibrium exists in this setting as well. This equilibrium can indeed be constructed by following exactly the same arguments employed in the stochastic case.

{{Games of optimal investment into production capacity, both in stochastic and deterministic settings, are extensively studied in the economics literature, and a comprehensive overview of models and results can be found in~\citet{vives1999oligopoly}. Specifically, mean-field problems with Cournot competition have attracted considerable attention in the recent literature; see \citet{chan.sircar.2017fracking} for an insightful overview. We now provide a non-exhaustive review of papers dealing with mean-field models of investment in capacity that are relevant for our setting.

\citet{graber.bensoussan.2018existence} study the existence (and, under certain conditions, uniqueness) of solutions to a system of partial differential equations (PDEs)—namely, the Hamilton--Jacobi--Bellman (HJB) and Fokker--Planck equations—associated with a mean-field game involving Bertrand and Cournot competition among a continuum of players. In~\citet{graber2023master}, existence and uniqueness of the master equation associated with a mean-field game of controls with absorption are established, while~\citet{chan.sircar.2015bertrand} analyze dynamic mean-field games with exhaustible capacities and interactions akin to Cournot and Bertrand competition. An optimal transport perspective on Cournot--Nash equilibria is explored in~\citet{acciaio2021cournot}.

The existence of classical solutions to an extended mean-field game arising in the study of exhaustible resource production under market competition is addressed in \citet{graber2021nonlocal}, where uniqueness of the equilibrium is shown to hold if the Hamiltonian satisfies certain smallness and uniform convexity conditions. Economically, this corresponds to imposing an upper bound on the degree of substitutability among competing firms.

\citet{camilli2025learning} also consider Cournot mean-field games of controls for the production of an exhaustible resource by a continuum of producers. They establish uniqueness of the solution under general assumptions on the price function and prove convergence of a learning algorithm, which in turn yields existence of a solution to the mean-field game system.

In \citet{escribe2024mean}, a stylized model for investment in renewable power plants under long-term uncertainty---arising from weather conditions, demand fluctuations, future fuel prices, and average national weather patterns---is considered. It is shown that the $N$-player formulation of the problem admits a Nash equilibrium and that any sequence of Nash equilibria of the $N$-player game converges to the unique solution of the corresponding mean-field game.

Investment in renewable energy production is also addressed in \citet{alasseur2023mean}, which proposes a model with a large number of small competitive producers to study the effect of subsidies on the aggregate level of capacity, taking into account a cannibalization effect. The analysis is based on the master equation, and explicit formulas for the long-run equilibria are obtained.

Related price formation models, in the context of large systems of small agents trading a commodity or a security, are studied in \citet{fujii2022mean, gomes2021mean} using mean-field techniques.

Regarding irreversible investment, we refer to \citet{cao2022:stationary} and \citet{aid2025stationary}. \citet{cao2022:stationary} consider stationary discounted and ergodic mean-field games of singular controls motivated by irreversible investment and provide existence and uniqueness results, as well as connections between the two classes of problems. Applications to commodity markets and the role of regime switches in equilibrium formation are investigated in \citet{aid2025stationary}, which studies a mean-field model of firms competing \`a la Cournot and irreversibly investing in production.}}

The rest of the paper is organized as follows. The stochastic mean-field game is introduced in Section~\ref{sec:model}, and in Subsection~\ref{sec:eco} we provide its economic foundation. In Section~\ref{sec:optpb} the company's optimal investment problem is solved for both the cases $T<+\infty$ and $T=+\infty$, while existence and uniqueness of the equilibrium is shown in Section~\ref{sec:fixed_point}, again for the finite and infinite time horizon cases. The deterministic version of the mean-field problem is finally considered in Section~\ref{sec:deterministic}.


\subsection{Notation}
\label{sec:notation}

In this section we collect the main notation used in this work.

\begin{itemize}
    \item Throughout the paper the set $\N$ denotes the set of natural integers without the zero element, i.e., $\N = \{1, 2, \dots \}$, while $\R$ denotes the set of real numbers.
Whenever $T=+\infty$, the notation $[0,T]$ indicates the interval $[0,+\infty).$
\item For any $p \geq 1$, any measure space $(E, \rE, \mu)$, and any interval $I \subseteq \R$, we indicate by $\dL^p((E, \rE, \mu); I)$ the set of all functions with values in $I$ that are $p$-integrable with respect to $\mu$. If $E \subseteq \R$ and it is Lebesgue-measurable, we take $\rE=\rL(E)$ as the Lebesgue $\sigma$-algebra and as $\mu$ the Lebesgue measure, which is denoted by $\mathrm{Leb}$, and we simply write $\dL^p(E;I)$. When computing integrals with respect to this measure, we simply write $\dd x$ instead of $\mathrm{Leb}(\dd x)$.

\item For any $\eta > 0$ and $E, \, I \subseteq \R$, the notation $\dL^p_\eta(E;I)$ indicates the set of all functions $f \colon E \to I$ such that $t \mapsto \de^{-\eta t} f(t)$ is $p$-integrable with respect to the Lebesgue measure.

\item 
Given two intervals $I, J \subseteq \R$, we denote by $\dC(I;J)$ the set of all continuous functions from $I$ to $J$. {If $I$ is compact, we endow this set with the usual sup-norm, which makes it a Banach space.}
The notation $\dC^p(I;J)$, $p \in \N \cup \{+\infty\}$, denotes the set of functions from $I$ to $J$ that are continuously differentiable $p$ times.
\end{itemize}

\section{The stochastic model}
\label{sec:model}

Let $(\Omega, \cF, \bbF \coloneqq (\cF_s)_{s \geq 0}, \bbP)$ be a complete filtered probability space, with $\bbF$ satisfying the usual assumptions, supporting
an $\bbF$-adapted standard Brownian motion $B$ and an independent \mbox{$\cF_0$-measurable} random variable $\xi$. Throughout the paper, $T \in (0, +\infty) \cup \{+\infty\}$ denotes a finite or infinite time horizon.

We consider a real-valued $\bbF$-adapted process $X = (X_s)_{s \in [0,T]}$, satisfying the stochastic differential equation (SDE)
\begin{equation}
\label{eq:SDE}
\left\{
\begin{aligned}
&\dd X_s = -\delta X_s \, \dd s + \sigma X_s \, \dd B_s + u_s \, \dd s, \quad s \in (0,T], \\
&X_0 = \xi \, ,
\end{aligned}
\right.
\end{equation}
where $\delta, \sigma > 0$ are given coefficients,
{and $\bfu \in \dL^1(\Omega \times [0,T]; [0,+\infty))$, i.e., $\bfu$ is a non-negative, $(\cF_s)_{s \in [0,T]}$-progressively measurable process satisfying $\bbE\left[\int_0^t u_s \, \dd s\right] < +\infty$, for all $t \in [0,T]$. Process $\bfu$ is the control, which we will choose in a suitable class of admissible processes, that will be specified below.}
Whenever necessary, to stress the dependence of the solution to~\eqref{eq:SDE} on the initial condition $\xi$ and on the control $\bfu$, we denote it by $X^{\xi, \bfu}$.

We observe that {SDE~\eqref{eq:SDE} admits a unique strong solution, which} can be written explicitly (cf.~\citep[Problem~5.6.15]{karatzas:brown}), for each {$\bfu \in \dL^1(\Omega \times [0,T]; [0,+\infty))$}, as
\begin{equation}
\label{eq:SDEsol}
X_s^{\xi,\bfu} = Y_s \xi + \int_0^s \dfrac{Y_s}{Y_r} u_r \, \dd r, \quad s \in [0,T],
\end{equation}
where
\begin{equation}
\label{eq:Y}
Y_s \coloneqq \de^{\sigma B_s - \left(\delta + \frac{\sigma^2}{2}\right) s}, \quad s \in [0,T].
\end{equation}

In the mean-field game studied in this paper, the process $X$ describes the private state of a representative player. In particular, it models the evolution of her/his production capacity, which depreciates at a rate $\delta$ and can be increased by choosing the investment rate $\bfu$. 
{Since we require process $\bfu$ to be non-negative, we are in a setting where investment actions are irreversible. From a mathematical point of view, this ensures useful properties of process $X$, which will be discussed below and used throughout the paper.}
Note {also} that the statistical distribution of the state of the other players in the economy does not influence the production capacity level of the representative player. 

Throughout the paper, we work under the following assumption.
\begin{assumption}
\label{hp:xi}
The initial condition $\xi$ of SDE~\eqref{eq:SDE} is an $\cF_0$-measurable, positive, and integrable random variable. More precisely, $\xi > 0$, $\bbP$-a.s., and $0 < \bbE[\xi] < +\infty$.
\end{assumption}



The assumption above ensures that, for any {$\bfu \in \dL^1(\Omega \times [0,T]; [0,+\infty))$}, {$X_s^{\xi, \bfu} > 0$}, $\bbP$-a.s., for all $s \in [0,T]$. Thus, the production capacity level of the representative agent is never negative.
Moreover, it grants us the following result, which will be used later on.

\begin{lemma}
\label{lem:X_finite_expectation}
Under Assumption~\ref{hp:xi} and for any {$\bfu \in \dL^1(\Omega \times [0,T]; [0,+\infty))$}, the unique solution to SDE~\eqref{eq:SDE} has finite first moment, given by
\begin{equation}
\label{eq:X_expectation}
\bbE[X^{\xi,\bfu}_s] = \bbE[\xi] \de^{-\delta s} + \bbE\left[\int_0^s \de^{-\delta(s-r)} u_r \, \dd r\right], \quad s \in [0,T].       
\end{equation}
\end{lemma}

\begin{proof}
Fix $\xi$ and $\bfu$ as above.
Using the expression of $X^{\xi, \bfu}$ given in~\eqref{eq:SDEsol} we have that
\begin{equation*}
\bbE[X^{\xi,\bfu}_s] = \bbE[Y_s \xi] + \bbE\left[\int_0^s \dfrac{Y_s}{Y_r} u_r\right] \, \dd r, \quad s \in [0,T].
\end{equation*}
Since the random variable $Y_s \xi$ is non-negative and $Y_s$ is independent of $\cF_0$, for all $s \in [0,T]$, we can directly compute the first summand
\begin{equation*}
\bbE[Y_s \xi] = \bbE[\xi] \bbE[Y_s] = \bbE[\xi] \de^{-\delta s}.
\end{equation*}
For the second summand, observe that the integrand is non-negative and that, for each fixed $s \in [0,T]$, the random variable $\frac{Y_s}{Y_r}$ is independent of $\cF_r$, for all $r \in [0,s]$. Therefore, applying the Fubini-Tonelli theorem,
\begin{equation*}
\bbE\left[\int_0^s \dfrac{Y_s}{Y_r} u_r \, \dd r\right]  = \int_0^s \bbE\left[\dfrac{Y_s}{Y_r} u_r\right] \dd r = \int_0^s \bbE\left[\dfrac{Y_s}{Y_r}\right] \bbE[u_r] \, \dd r = \bbE\left[\int_0^s \de^{-\delta(s-r)} u_r \, \dd r\right],
\end{equation*}
which is finite thanks to the assumptions on $\bfu$. Therefore, we get~\eqref{eq:X_expectation}.
\end{proof}

In our mean-field game, the representative company's instantaneous profit function is derived from the isoelastic demand function of Spence-Dixit-Stiglitz preferences (see \citet[p.~7, Footnote~5]{achdouetal2014:pdemodels} and \citet{luttmer2007:selection}), as explained in Section~\ref{sec:eco} below. Moreover, the company faces a quadratic cost of investment.
This leads to the expected net profit functional
\begin{equation}\label{eq:netprofitfunct}
J_{T,q}(\xi, \bfu) \coloneqq \bbE\left[\int_0^T \de^{-\rho s} \left\{X_s^{\xi,\bfu} q_s^{-\beta} - \dfrac 12 u_s^2\right\} \dd s\right] \,.
\end{equation}
Here, $\rho > 0$ is a discount factor, $\beta > 0$ is a fixed parameter, and $q = (q_s)_{s \in [0,T]}$ is a given deterministic measurable function. 
{
The control process $\bfu$ is chosen in either of the following two classes of admissible controls: if $T < +\infty$,
\begin{align}
\label{eq:admctlr_finite}
  \cU_{T} \coloneqq &\biggl\{\bfu = (u_s)_{s \in [0,T]} \text{ s.t. } \bfu \colon \Omega \times [0,T] \to [0,+\infty) \text{ is } (\cF_s)_{s \in [0,T]}\text{-progressively measurable and } \notag
\\
&\qquad \bbE\biggl[\int_0^T u_s^2 \, \dd s\biggr] < +\infty\biggr\} \, ;
\end{align}
if, instead, $T = +\infty$,
\begin{align}
\label{eq:admctlr_infinite}
  \cU_{\infty} \coloneqq
	&\biggl\{\bfu = (u_s)_{s \geq 0} \text{ s.t. } \bfu \colon \Omega \times [0,\infty) \to [0,+\infty) \text{ is } \bbF\text{-progressively measurable,} \notag
	\\
	&\qquad \bbE\left[\int_0^t u_s \, \dd s\right] < +\infty, \, \bbP\text{-a.s., } \forall t \geq 0, \text{ and } \bbE\biggl[\int_0^{+\infty} \de^{-\rho s} u_s^2 \, \dd s\biggr] < +\infty\biggr\} \, .
\end{align}%
}%
At equilibrium, the function $q$ will identify with the average production capacity of the whole population of agents. This is formalized in the next definition.


\begin{definition}
\label{def:equilibrium}
Fix a random variable $\xi$ under Assumption~\ref{hp:xi}.

A pair $(\widehat \bfu, \widehat{q})$, where $\widehat \bfu \in \cU_T$ and $\widehat{q} \colon [0,T] \to (0,+\infty)$ is a measurable function, is an \emph{equilibrium} of the mean-field game if
\begin{enumerate}[label=(\roman*)]
\item \label{def:equilibrium:opt} 
$J_{T,\widehat{q}}(\xi, \widehat \bfu) \geq J_{T,\widehat{q}}(\xi, \bfu)$, for all $\bfu \in \cU_T$;
\item \label{def:equilibrium:mean}
$\widehat{q}_s = \bbE[X_s^{\xi,\widehat \bfu}]$, for all $s \in [0,T]$.
\end{enumerate}
\end{definition}

{From SDE~\eqref{eq:SDE}, the expression of the expected net profit functional in~\eqref{eq:netprofitfunct}, and the definition of equilibrium given above, we see that the mean-field interaction between firms appears only in the functional to be maximized, and not also in the state equation.
We also emphasize that the interaction is only through the average production capacity, thus our problem is set in a \mbox{Cournot-like} competition framework and, as we are going to see, the optimal control for the representative firm is going to depend on this quantitity.

The very specific nature of our problem allows us to tackle it with \textit{ad hoc} arguments, as we outline below, rather than using either the standard analytic approach (see, e.g.,~\citep{cardaliaguet:notes}), where a system of coupled Hamilton-Jacobi-Bellman and Kolmogorov-Fokker-Plank equations is studied, or the standard probabilistic approach (see, e.g.,~\citep{carmona2018probabilistic}), where a system of Forward Backward Stochastic Differential Equations is considered.}

In the next sections we show that there exists a unique equilibrium by adopting a classic fixed point approach.
The first step is to solve an optimization problem for a given measurable function $q$, representing the evolution of the average production capacity level of the agents in the economy. We show that there exists a unique explicit and deterministic optimal control $\widehat \bfu \in \cU_T$, depending on $q$, which provides us with the best response of the representative agent to the distribution of the states of the other agents in the economy, summarized by the average $q$.

The second step is to determine the optimally controlled dynamics of the production capacity level of the representative agent. Thanks to~\eqref{eq:X_expectation}, we can compute explicitly $\bbE[X^{\xi, \widehat \bfu}_s]$, $s \in [0,T]$. Since the optimal control $\widehat \bfu$ depends on $q$, the fixed point argument follows from condition~\ref{def:equilibrium:mean} in Definition~\ref{def:equilibrium}. We show that proving that there exists a unique equilibrium reduces to finding the unique solution of an integro-differential equation.

\subsection{The economic foundation}
\label{sec:eco}
To provide an economic foundation of our model, let us introduce the $N$-player game corresponding to the mean-field game introduced above.

We consider a finite number of firms $N$, whose production capacities evolve stochastically over time according to the following dynamics
\begin{equation}
\label{eq:SDE_N}
\left\{
\begin{aligned}
&\dd X^i_s = -\delta X^i_s \, \dd s + \sigma X^i_s \, \dd B^i_s + u^i_s \, \dd s, \quad s \in (0,T], \\
&X^i_0 = \xi^i \, ,
\end{aligned}
\right.
\end{equation}
where $B^1, \dots, B^n$ are independent Wiener processes and $\xi^1, \dots, \xi^N$ are i.i.d., $\cF_0$-measurable, positive, and integrable random variables, independent from the Brownian motions $B^i$, $i=1,\dots,n$. 

Each firm may increase its own production capacity, by investing at rate $\bfu^i = (u^i_s)_{s \in [0,T]} \in \cU_T$, and aims at maximizing the expected net profit functional
\begin{equation}
\label{eq:netprofitfunct_N}
\bbE\Biggl[\int_0^T \de^{-\rho s} \Biggl\{X^i_s \biggl(\frac 1N \sum_{j=1}^N X^j_s\biggr)^{-\beta} - \dfrac 12 (u^i_s)^2\Biggr\} \dd s\Biggr] \,.
\end{equation}

Taking the limit as $N \to \infty$ and considering Nash equilibria, we expect that the dynamics~\eqref{eq:SDE_N} of each firm converges to the dynamics~\eqref{eq:SDE} of the representative player of the mean-field game, and that at equilibrium the functional~\eqref{eq:netprofitfunct_N} converges to $J_{T, \widehat q}$ given in~\eqref{eq:netprofitfunct}, with $\widehat q$ satisfying the consistency condition of Definition~\ref{def:equilibrium}-\ref{def:equilibrium:mean}.

We leave for future research the study of the $N$-player game and its convergence to the corresponding mean-field game, since this requires a detailed and careful analysis which is outside the scope of this paper, as we intend to focus on directly studying the mean-field game.

Now, instead, we want to focus on providing the precise derivation of the instantaneous profit function appearing in~\eqref{eq:netprofitfunct_N}, i.e., the function 
\begin{equation}
\label{eq:profit_N}
    (x_1, \dots, x_n) \mapsto x_i \biggl(\frac 1N \sum_{j=1}^N {x_j}\biggr)^{-\beta},    
\end{equation}
for any fixed $i = 1, \dots, N$. This in turn justifies the form of the profit function $(x,p) \mapsto xp^{-\beta}$ appearing in~\eqref{eq:netprofitfunct}.
We follow an approach close to the one discussed in~\citet[p.~7, Footnote~5]{achdouetal2014:pdemodels}.

Assume that each player $i = 1, \dots, N$ in our economy is a firm producing a single good and facing the demand function
\begin{equation*}
\left(\frac{\pi_i}{P}\right)^{-\gamma}, \quad \text{with } \gamma > 1,
\end{equation*}
where $\gamma$ is the elasticity of substitution, {$\pi_i$} is the price set by firm $i$ and $P = \left(\frac 1N \sum_{j=1}^N \pi_j^{1-\gamma}\right)^{1/(1-\gamma)}$ is a \textit{price index}.
This choice of the demand function is standard in economics and is justified by the Spence-Dixit-Stiglitz preferences (see, e.g.,~\citet{dixit1977monopolistic}).

Each firm aims at maximizing the function
\begin{equation*}
\pi_i \mapsto \left\{\pi_i \left(\frac{\pi_i}{P}\right)^{-\gamma} - \frac{1}{x_i^{{\eta}}} \left(\frac{\pi_i}{P}\right)^{-\gamma} \right\},
\end{equation*}
where $x_i > 0$ is the productivity (or size) of firm $i$ and, hence, costs decrease with respect to productivity, according to the factor $x_i^{-{\eta}}$, where {$\eta > 0$} is a given parameter. We can consider each firm production capacity as a proxy for $x_i$ and, thus, this setting is coherent with the model introduced above. 

It is easy to see that the maximum profit is 
\begin{equation}
\label{eq:max_profit_N}
\frac{1}{\gamma-1} \left(\frac{\gamma}{\gamma - 1}\right)^{-\gamma} x_i^{{\eta}(\gamma-1)} P^\gamma \, ,
\end{equation}
which is attained setting the price $\pi_i = \frac{\gamma}{(\gamma-1)}x_i^{-{\eta}}$.

This entails that $P = \frac{\gamma}{(\gamma-1)} \left(\frac 1N \sum_{j=1}^N x_j^{{\eta}(\gamma-1)} \right)^{-1/(\gamma-1)}$, and hence, substituting this expression into~\eqref{eq:max_profit_N}, we get that the profit is
\begin{equation}
\label{eq:profit_general_N}
\frac{1}{\gamma-1} x_i^{{\eta}(\gamma-1)}\left(\frac 1N \sum_{j=1}^N x_j^{{\eta}(\gamma-1)}\right)^{-\frac{\gamma}{\gamma-1}}\, .
\end{equation}
{Therefore, setting $\beta = \frac{\gamma}{\gamma-1}$, $\eta = \frac{1}{\gamma-1}$,
we get the profit function~\eqref{eq:profit_N}}, except for the constant $\frac{1}{\gamma-1}$ which does not affect the optimization problem.
{We emphasize once more that~\eqref{eq:profit_N} justifies the form of the profit function $(x,p) \mapsto xp^{-\beta}$ appearing in~\eqref{eq:netprofitfunct}. Indeed, we formally substitute the production capacity $x_i$ of firm $i$ with the production capacity $x$ of the representative firm, and we replace $\frac 1N \sum_{j=1}^N x_j$ with variable $p$ which, at equilibrium, will coincide with the average production capacity of the mean-field game.} 



The discussion above highlights that it is possible to consider more general forms of interaction between players in the profit function, which can be taken of the form
\begin{equation}
\label{eq:profitmoregeneral}
\Phi(x)\, F\left(\frac 1N \sum_{j=1}^N f(x_j)\right),
\end{equation}
for some $\Phi, f \colon \R \to [0,+\infty)$ strictly increasing and $F \colon [0,+\infty) \to [0,+\infty)$, see, e.g.,~\citep{cao2022:stationary}.
However, with the aim of obtaining an explicit characterization of the equilibrium, in this paper we focus on the particular case in which $\Phi$ and $f$ are the identity and $F(x)=x^{-\beta}$, as in~\eqref{eq:profit_general_N}. Nonetheless, it is worth noting that part of our subsequent results -- in particular the existence ones -- could be extended to the more general case of profit function~\eqref{eq:profitmoregeneral}, upon exploiting the monotonicity of functions $\Phi, F,$ and $f$.

\section{The optimization problem}
\label{sec:optpb}
In this section we consider the optimization problem associated to the mean-field game described in Section~\ref{sec:model}. We show that, for each possible choice (in a suitable class of measurable functions) of the function $q$ appearing in the discounted net profit functional~\eqref{eq:netprofitfunct}, the corresponding maximization problem has a unique solution and we compute explicitly the associated optimal control.

%

We divide our analysis into two subsections, the first devoted to the finite time horizon case, the second one to the infinite time horizon case.
{Since an additional technical condition is required to solve the infinite time horizon case, we discuss the two problems separately, for the sake of clarity.}

\subsection{The finite time horizon case}
\label{sec:optpb:finite}
Let us consider the finite time horizon case, i.e., fix $T < +\infty$.
{We introduce the set
\begin{equation}\label{eq:QT}
    \cQ_T \coloneqq \{q \colon [0,T] \to (0,+\infty) \text{ s.t. } q^{-\beta} \in \dL^1((0,T]; (0,+\infty))\} \, .
\end{equation}
}
For each fixed {$q \in \cQ_T$}, we consider the problem of maximizing the functional
\begin{equation}
\label{eq:J_finite}
J_{T,q}(\xi, \bfu) \coloneqq \bbE\left[\int_0^T \de^{-\rho s} \left\{X^{\xi,\bfu}_s q_s^{-\beta} - \dfrac 12 u_s^2\right\} \dd s\right],
\end{equation}
where $\xi$ satisfies Assumption~\ref{hp:xi}, $\bfu$ is chosen in the class of admissible controls introduced in~\eqref{eq:admctlr_finite},
and $X^{\xi,\bfu}$ is the unique solution of SDE~\eqref{eq:SDE}.
We also introduce the value function corresponding to the optimization problem, namely,
\begin{equation}
\label{eq:V_finite}
V_{T,q}(\xi) \coloneqq \sup_{\bfu \in \cU_{T}} J_{T,q}(\xi, \bfu), \quad \xi \in \dL^1((\Omega, \cF_0, \bbP); (0,+\infty)).
\end{equation}

In the next {Lemma} we are going to show important properties of the functional $J_{T,q}$.

\begin{lemma}
\label{prop:J_finite_properties}
Fix a random variable $\xi$ satisfying Assumption~\ref{hp:xi}, $\bfu \in \cU_{T}$, and {$q \in \cQ_T$}.
Then, the functional $J_{T,q}$ defined in~\eqref{eq:J_finite} is finite and verifies
\begin{equation}
\label{eq:J_finite_rewrite}
J_{T,q}(\xi,\bfu) = \bbE[\xi] z^{(T,q)}_0 + \bbE\left[\int_0^T \de^{-\rho s} \left\{ z^{(T,q)}_s u_s - \dfrac 12 u_s^2\right\} \dd s\right],
\end{equation}
where $z^{(T,q)} \colon [0,T] \to [0,+\infty)$ is the deterministic function given by
\begin{equation}
\label{eq:z_finite}
z^{(T,q)}_s \coloneqq \int_s^T \de^{-(\rho+\delta) (r-s)} q_r^{-\beta} \, \dd r, \quad s \in [0,T].
\end{equation}
Moreover, if $\xi, \xi^\prime$ both verify Assumption~\ref{hp:xi} and are such that $\bbE[\xi] = \bbE[\xi^\prime]$, then $J_{T,q}(\xi;\bfu) = J_{T,q}(\xi^\prime;\bfu)$, for any $\bfu \in \cU_{T}$.
\end{lemma}

\begin{proof}
Fix $\xi \in \dL^1((\Omega, \cF_0, \bbP); (0,+\infty))$, $\bfu \in \cU_{T}${, and $q \in \cQ_T$.}
Since the assumptions of Lemma~\ref{lem:X_finite_expectation} are verified, we can use~\eqref{eq:X_expectation} and apply the Fubini-Tonelli theorem (note that the integrand is non-negative) to get
\begin{equation*}
\bbE\left[\int_0^T \de^{-\rho s} X^{\xi,\bfu}_s q_s^{-\beta} \, \dd s \right]
= \bbE[\xi] \int_0^T \de^{-\rho s} \de^{-\delta s} q_s^{-\beta} \, \dd s + \bbE\left[\int_0^T \int_0^s \de^{-\rho s} \de^{-\delta(s-r)} q_s^{-\beta} u_r \, \dd r \, \dd s \right].
\end{equation*}

Clearly, the first summand of the last equality is finite, thanks to the assumptions on $\xi$ and $q$. Moreover,
\begin{equation*}
\bbE[\xi] \int_0^T \de^{-\rho s} \de^{-\delta s} q_s^{-\beta} \, \dd s = \bbE[\xi] z^{(T,q)}_0.
\end{equation*}
Also the second summand is finite, thanks to the assumptions on $q$ and $\bfu$. Exchanging the order of the two time integrals we get
\begin{align*}
&\bbE\left[\int_0^T \int_0^s \de^{-\rho s} \de^{-\delta(s-r)} q_s^{-\beta} u_r \, \dd r \, \dd s \right]
= \bbE\left[\int_0^T \de^{\delta r} u_r \int_r^T \de^{-(\rho+\delta) s} q_s^{-\beta} \, \dd s \, \dd r \right]
\\
= &\bbE\left[\int_0^T \de^{-\rho r} u_r \int_r^T \de^{-(\rho+\delta) (s-r)} q_s^{-\beta} \, \dd s \, \dd r \right]
= \bbE\left[\int_0^T \de^{-\rho r} z^{(T,q)}_r u_r \, \dd r \right].
\end{align*}
It easily follows that $J_{T,q}$ is finite and that Equation~\eqref{eq:J_finite_rewrite} holds, which also entails the last statement of the {Lemma}.
\end{proof}


\begin{remark}
\label{rem:V_average_finite}
The last statement in {Lemma}~\ref{prop:J_finite_properties} implies that the value function $V_{T,q}$ depends on the initial condition $\xi$ only through its average. More precisely, if $\xi, \xi^\prime$ both verify Assumption~\ref{hp:xi} and are such that $\bbE[\xi] = \bbE[\xi^\prime]$, then $V_{T,q}(\xi) = V_{T,q}(\xi^\prime)$.
\end{remark}

Thanks to the {Lemma} above we can find the optimal control for our optimization problem and this, as a byproduct, allows us to explicitly compute the value function.
\begin{proposition}
\label{th:V_finite}
Fix a random variable $\xi$ satisfying Assumption~\ref{hp:xi} and {$q \in \cQ_T$}.
Then, $\widehat \bfu^{(T,q)} \coloneqq z^{(T,q)} \in \cU_{T}$, where $z^{(T,q)}$ is the function defined in~\eqref{eq:z_finite}, is an optimal control for problem~\eqref{eq:V_finite}, which is deterministic and independent of $\xi$.

Moreover, $\widehat \bfu^{(T,q)}$ is essentially unique, i.e., if $\overline \bfu^{(T,q)} \in \cU_{T}$ is an optimal control for problem~\eqref{eq:V_finite} different from $\widehat \bfu^{(T,q)}$, then
\begin{equation*}
\overline u^{(T,q)}_s = \widehat u^{(T,q)}_s, \quad \text{for $\bbP \otimes \mathrm{Leb}$-a.e. } (\omega,s) \in \Omega \times [0,T].
\end{equation*}

Finally, the value of the optimization problem admits the explicit expression
\begin{equation}
\label{eq:V_finite_explicit}
V_{T,q}(\xi) = \bbE[\xi] z^{(T,q)}_0 + \dfrac{1}{2} \int_0^T \de^{-\rho s} (z^{(T,q)}_s)^2 \, \dd s,
\end{equation}
and the optimally controlled state process $X^{\xi, \widehat \bfu^{(T,q)}}$ is given by
\begin{equation}
\label{eq:X_opt_ctrl_finite}
X^{\xi, \widehat \bfu^{(T,q)}}_s = Y_s \xi + \int_0^s \dfrac{Y_s}{Y_r} z^{(T,q)}_r \, \dd r, \quad s \in [0,T] \, ,
\end{equation}
where $Y$ is the process defined in~\eqref{eq:Y}.
\end{proposition}

\begin{proof}
Fix $\xi \in \dL^1((\Omega, \cF_0, \bbP); (0,+\infty))$ {$q \in \cQ_T$.}
From~\eqref{eq:J_finite_rewrite} it immediately follows that
\begin{equation}
\label{eq:V_finite_proof}
V_{T,q}(\xi) = \bbE[\xi] z^{(T,q)}_0 + \sup_{\bfu \in \cU_{T}} \bbE\left[\int_0^T \de^{-\rho s} \left\{ z^{(T,q)}_s u_s - \dfrac 12 u_s^2\right\} \dd s\right].
\end{equation}
Hence, if we can find $\widehat \bfu^{(T,q)} \in \cU_{T}$ that maximizes the integrand $z^{(T,q)}_s u_s - \dfrac 12 u_s^2$, for $\bbP \otimes \mathrm{Leb}$-a.e. $(\omega,s) \in \Omega \times [0,T]$, then $\widehat \bfu^{(T,q)}$ it must be optimal.
Clearly, the integrand is maximized in the sense above if we take $\widehat \bfu^{(T,q)} = z^{(T,q)}$. Its admissibility is a consequence of the fact that $z^{(T,q)}$ is bounded. Therefore, $\widehat \bfu^{(T,q)}$ is optimal, and clearly essentially unique in the sense specified above.
Substituting its definition in~\eqref{eq:V_finite_proof} we immediately get~\eqref{eq:V_finite_explicit} and~\eqref{eq:X_opt_ctrl_finite}.
\end{proof}

\subsection{The infinite time horizon case}
\label{sec:optpb:infinite}
We analyze now the infinite time horizon case, i.e., we set $T = +\infty$.
{Let us define, for any given measurable $q \colon [0,+\infty) \to (0,+\infty)$ the function
\begin{equation}
\label{eq:z_infinite}
z^{(\infty,q)}_s \coloneqq \int_s^{+\infty} \de^{-(\rho+\delta) (r-s)} q_r^{-\beta} \, \dd r \, , \quad s \geq 0 \, .
\end{equation}
We introduce the set
\begin{equation}\label{eq:Qinfty}
    \cQ_\infty \coloneqq \{q \colon [0,+\infty) \to (0,+\infty) \text{ s.t. } q^{-\beta} \in \dL^1((0,+\infty); (0,+\infty)) \text{ and } z^{(\infty,q)} \text{ is bounded} \} \, .
\end{equation}
Note that in the definition of $\cQ_\infty$ we require an additional technical condition with respect to the definition of $\cQ_T$ in the finite time horizon case. This is needed to prove the results of this subsection, namely, Lemma~\ref{prop:J_infinite_properties} and Proposition~\ref{th:V_infinite}.}

For each fixed {$q \in \cQ_\infty$} we consider the problem of maximizing the functional
\begin{equation}
\label{eq:J_infinite}
J_{\infty,q}(\xi, \bfu) \coloneqq \bbE\left[\int_0^{+\infty} \de^{-\rho s} \left\{X^{\xi,\bfu}_s q_s^{-\beta} - \dfrac 12 u_s^2\right\} \dd s\right],
\end{equation}
where $\xi$ satisfies Assumption~\ref{hp:xi}, $\bfu$ is chosen in the class of admissible controls defined in~\eqref{eq:admctlr_infinite},
and $X^{\xi,\bfu}$ is the unique solution of SDE~\eqref{eq:SDE}.
We also introduce the value function corresponding to the optimization problem, namely,
\begin{equation}
\label{eq:V_infinite}
V_{\infty,q}(\xi) \coloneqq \sup_{\bfu \in \cU_{\infty}} J_{\infty,q}(\xi, \bfu), \quad \xi \in \dL^1((\Omega, \cF_0, \bbP); (0,+\infty)).
\end{equation}

The next result is analogous to {Lemma}~\ref{prop:J_finite_properties}. The proof proceeds along the same lines and, thus, we omit it.
\begin{lemma}
\label{prop:J_infinite_properties}
Fix a random variable $\xi$ satisfying Assumption~\ref{hp:xi}, $\bfu \in \cU_{\infty}${, and $q \in \cQ_\infty$.}

Then, the functional $J_{\infty,q}$ defined in~\eqref{eq:J_infinite} is finite and verifies
\begin{equation}
\label{eq:J_infinite_rewrite}
J_{\infty,q}(\xi, \bfu) = \bbE[\xi] z^{(\infty,q)}_0 + \bbE\left[\int_0^{+\infty} \de^{-\rho s} \left\{ z^{(\infty,q)}_s u_s - \dfrac 12 u_s^2\right\} \dd s\right].
\end{equation}

Moreover, if $\xi, \xi^\prime$ both verify Assumption~\ref{hp:xi} and are such that $\bbE[\xi] = \bbE[\xi^\prime]$, then $J_{\infty,q}(\xi, \bfu) = J_{\infty,q}(\xi^\prime, \bfu)$, for any $\bfu \in \cU_{\infty}$.
\end{lemma}

\begin{remark}
\label{rem:V_average_infinite}
Also in this case (cf. Remark~\ref{rem:V_average_finite}) the last statement in {Lemma}~\ref{prop:J_infinite_properties} entails that the value function $V_{\infty,q}$ depends on the initial condition $\xi$ only through its average.
\end{remark}

Thanks to {Lemma}~\ref{prop:J_infinite_properties}, also in the infinite time horizon case we can find the optimal control for problem~\eqref{eq:V_infinite} and the explicit form of the corresponding value function, as stated in the next {Proposition}. Its proof is omitted, being similar to that of {Proposition}~\ref{th:V_finite}.
\begin{proposition}
\label{th:V_infinite}
Fix a random variable $\xi$ satisfying Assumption~\ref{hp:xi} {and $q \in \cQ_\infty$.}

Then, $\widehat \bfu^{(\infty,q)} \coloneqq z^{(\infty,q)} \in \cU_{\infty}$ is an optimal control for problem~\eqref{eq:V_infinite}, which is deterministic and independent of $\xi$.

Moreover, $\widehat \bfu^{(\infty,q)}$ is essentially unique, i.e., if $\overline \bfu^{(\infty,q)} \in \cU_{\infty}$ is an optimal control for problem~\eqref{eq:V_infinite} different from $\widehat \bfu^{(\infty,q)}$, then
\begin{equation*}
\overline u^{(\infty,q)}_s = \widehat u^{(\infty,q)}_s, \quad \text{for $\bbP \otimes \mathrm{Leb}$-a.e. } (\omega,s) \in \Omega \times [0,+\infty).
\end{equation*}

Finally, the value function of the optimization problem admits the explicit expression
\begin{equation}
\label{eq:V_infinite_explicit}
V_{\infty,q}(\xi) = \bbE[\xi] z^{(\infty,q)}_0 + \dfrac{1}{2} \int_0^{+\infty} \de^{-\rho s} (z^{(\infty,q)}_s)^2 \, \dd s,
\end{equation}
and the optimally controlled state process $X^{\xi, \widehat \bfu^{(\infty,q)}}$ is given by
\begin{equation}
\label{eq:X_opt_ctrl_infinite}
X^{\xi, \widehat \bfu^{(\infty,q)}}_s = Y_s \xi + \int_0^s \dfrac{Y_s}{Y_r} z^{(\infty,q)}_r \, \dd r, \quad s \geq 0,
\end{equation}
where $Y$ is the process defined in~\eqref{eq:Y}.
\end{proposition}

\begin{remark}\label{rem:extension_t}
All the results of this section can be straightforwardly extended to the case where the optimization problem starts at a time $t > 0$. More precisely, setting the initial condition $X_t = x > 0$ in SDE~\eqref{eq:SDE}, we get that~\eqref{eq:X_expectation} becomes
\begin{equation*}
\bbE[X^{\xi,\bfu}_s] = \bbE[\xi] \de^{-\delta(s-t)} + \bbE\left[\int_t^s \de^{-\delta(s-r)} u_r \, \dd r\right], \quad s \in [t,T],
\end{equation*}
while the expressions of the functional $J_{T,q}$ and of the value function $V_{T,q}$ (now both dependent on $t$) given in~\eqref{eq:J_finite_rewrite} and~\eqref{eq:V_finite_explicit}, if $T < +\infty$, and in~\eqref{eq:J_infinite_rewrite} and~\eqref{eq:V_infinite_explicit}, if $T = +\infty$, become
\begin{gather*}
J_{T,q}(t,\xi, \bfu) = \bbE[\xi] \de^{-\rho t} z^{(T,q)}_t + \bbE\left[\int_t^T \de^{-\rho s} \left\{ z^{(T,q)}_s u_s - \dfrac 12 u_s^2\right\} \dd s\right], \\
V_{T,q}(t,\xi) = \bbE[\xi] \, \de^{-\rho t} z^{(T,q)}_t + \dfrac{1}{2} \int_t^T \de^{-\rho s} (z^{(T,q)}_s)^2 \, \dd s,
\end{gather*}
with optimal control $\widehat \bfu^{(T,q)} \coloneqq z^{(T,q)}$, where
\begin{equation*}
z^{(T,q)}_s \coloneqq \int_s^T \de^{-(\rho+\delta) (r-s)} q_r^{-\beta} \, \dd r, \quad s \in [t,T].
\end{equation*}
\end{remark}


\section{Existence and uniqueness of equilibria}
\label{sec:fixed_point}

In this section we discuss the existence and the uniqueness of an equilibrium for the mean-field game introduced in Section~\ref{sec:model}.

As usual in mean-field games theory, looking for an equilibrium boils down to finding a fixed point of a suitable map. According to Definition~\ref{def:equilibrium}, in our case this map is $(q_s)_{s \in [0,T]} \mapsto \left(\bbE[X_s^{\xi, \widehat \bfu^{(T,q)}}]\right)_{s \in [0,T]}$, where $\widehat \bfu^{(T,q)}$ is the optimal control for the maximization problem studied in Section~\ref{sec:optpb}, whose expression is given either in {Proposition}~\ref{th:V_finite}, in the case $T < +\infty$, or in {Proposition}~\ref{th:V_infinite}, in the case $T = +\infty$.

The next result formalizes this fact, showing also that the fixed point map is precisely the solution map of an integral equation.
\begin{theorem}
\label{th:MF_IDE}
Let us fix a random variable $\xi$ satisfying Assumption~\ref{hp:xi}.

Consider the mean-field game introduced in Section~\ref{sec:model} in the finite time horizon case, i.e., $T < +\infty$. Then,
\begin{enumerate}[label = (\roman*)]
\item\label{th:MF_IDE_finite_if} If there exists an equilibrium $(\widehat \bfu, \widehat{q})$ in the sense of Definition~\ref{def:equilibrium}, such that {$\widehat q \in \cQ_T$,}
then $\widehat{q}$ is a solution to the integral equation
\begin{equation}\label{eq:IDEint_finite}
y_s =\de^{-\delta s} \bbE[\xi] +\int_0^s \de^{-\delta (s-r)} \int_r^T \de^{-(\rho+\delta)(u-r)} y_u^{-\beta} \, \dd u \, \dd r,  \quad s \in [0,T] \, .
\end{equation}
\item\label{th:MF_IDE_finite_onlyif} Vice versa, if there exist a unique solution {$\widehat{q} \in \cQ_T$} to~\eqref{eq:IDEint_finite}, then there exists a unique equilibrium $(\widehat \bfu, \widehat{q}) = (z^{(T,\widehat{q})},\widehat{q})$ of the mean-field game among all equilibria {$(\widetilde \bfu, \widetilde q) \in \cU_T \times \cQ_T$}, where $z^{(T,q)}$ is the function defined in~\eqref{eq:z_finite}.
\end{enumerate}
Consider, instead, the same mean-field game in the infinite time horizon case, i.e., $T = +\infty$. Then,
\begin{enumerate}[label = (\roman*)]\setcounter{enumi}{2}
\item\label{th:MF_IDE_infinite_if} If there exists an equilibrium $(\widehat \bfu, \widehat{q})$ in the sense of Definition~\ref{def:equilibrium}, such that {$\widehat q \in \cQ_\infty$},
then $\widehat{q}$ is a solution to the integral equation
\begin{equation}\label{eq:IDEint_infinite}
y_s =\de^{-\delta s} \bbE[\xi] +\int_0^s \de^{-\delta (s-r)} \int_r^{+\infty} \de^{-(\rho+\delta)(u-r)} y_u^{-\beta} \, \dd u \, \dd r,  \quad s \geq 0 \, .
\end{equation}
\item\label{th:MF_IDE_infinite_onlyif} Vice versa, if there exist a unique solution {$\widehat{q} \in \cQ_\infty$} to~\eqref{eq:IDEint_infinite}, then there exists a unique equilibrium $(\widehat \bfu, \widehat{q}) = (z^{(\infty,\widehat{q})},\widehat{q})$ of the mean-field game among all equilibria {$(\widetilde \bfu, \widetilde q) \in \cU_\infty \times \cQ_\infty$, where $z^{(T,q)}$ is the function defined in~\eqref{eq:z_finite}.}
\end{enumerate}
\end{theorem}

\begin{proof}
We prove only the first two statements in the finite time horizon case, as the case $T = +\infty$ is analogous.

\medskip

\noindent \textit{\ref{th:MF_IDE_finite_if}.} Let $(\widehat \bfu, \widehat{q})$ be an equilibrium of the mean-field game, with $\widehat \bfu \in \cU_T$ and $\widehat{q} \colon [0,T] \to (0,+\infty)$, such that $\widehat{q}^{-\beta} \in \dL^1((0,T] ; (0,+\infty))$.
Since point~\ref{def:equilibrium:opt} of Definition~\ref{def:equilibrium} holds, from {Proposition}~\ref{th:V_finite} we have that $\widehat \bfu = z^{(T,\widehat{q})}$, for $\bbP \otimes \mathrm{Leb}$-a.e.  $(\omega,s) \in \Omega \times [0,T]$. Combining the expression of $z^{(T,q)}$, given in~\eqref{eq:z_finite}, and~\eqref{eq:X_expectation}, we get
\begin{align}
\bbE[X^{\xi,\widehat \bfu}_s] 
&= \bbE[\xi] \de^{-\delta s} + \bbE\left[\int_0^s \de^{-\delta(s-r)} \widehat u_r \, \dd r\right] = \bbE[\xi] \de^{-\delta s} + \bbE\left[\int_0^s \de^{-\delta(s-r)} z^{(T,\widehat{q})}_r \, \dd r\right] \notag
\\
&= \de^{-\delta s} \bbE[\xi] +\int_0^s \de^{-\delta (s-r)} \int_r^T \de^{-(\rho+\delta)(u-r)} \widehat{q}_u^{-\beta} \, \dd u \, \dd r, \quad s \in [0,T].  \label{eq:X_expect_optimal}  
\end{align}
Therefore, by point~\ref{def:equilibrium:mean} of Definition~\ref{def:equilibrium}, we get that $\widehat{q}$ is a solution to~\eqref{eq:IDEint_finite}.

\medskip

\noindent \textit{\ref{th:MF_IDE_finite_onlyif}.} Let $\widehat{q}$ be the unique solution to~\eqref{eq:IDEint_finite} such that $\widehat{q}^{-\beta} \in \dL^1((0,T] ; (0,+\infty))$.
Define $\widehat \bfu \coloneqq z^{(T, \widehat{q})}$ according to~\eqref{eq:z_finite}. Then, by {Proposition}~\ref{th:V_finite}, we know that $\widehat\bfu$ is optimal for the problem of maximizing $J_{T, \widehat{q}}$, i.e., that point~\ref{def:equilibrium:opt} of Definition~\ref{def:equilibrium} holds.
We also have that equation~\eqref{eq:X_expect_optimal} above is verified, and hence, since we assumed that $\widehat{q}$ solves~\eqref{eq:IDEint_finite}, we get that $\widehat{q}_s = \bbE[X^{\xi, \widehat \bfu}_s]$, for all $s \in [0,T]$, i.e., that point~\ref{def:equilibrium:mean} of Definition~\ref{def:equilibrium} also holds.
This means that $(\widehat \bfu, \widehat{q})$ is an equilibrium of the mean-field game.

If $(\widetilde \bfu, \widetilde q)$ is any other equilibrium of the mean-field game, such that $\widetilde q^{-\beta} \in \dL^1((0,T] ; (0,+\infty))$, then necessarily $\widetilde \bfu = z^{(T, \widetilde q)}$, for $\bbP \otimes \mathrm{Leb}$-a.e.  $(\omega,s) \in \Omega \times [0,T]$, by {Proposition}~\ref{th:V_finite}. Therefore, we can repeat the same argument of the proof of point~\ref{th:MF_IDE_finite_if} to get that $\widetilde q$ is a solution to~\eqref{eq:IDEint_finite}, and hence that $(\widetilde \bfu, \widetilde q) = (\widehat \bfu, \widehat{q})$, since we assumed that~\eqref{eq:IDEint_finite} has a unique solution.
\end{proof}

Thanks to Theorem~\ref{th:MF_IDE} we can reduce our search for a unique equilibrium of our mean-field game to proving existence and uniqueness of a solution to the integral equation~\eqref{eq:IDEint_finite}, in the finite time horizon case, or~\eqref{eq:IDEint_infinite}, in the infinite time horizon case. This, in turn, is equivalent to showing existence and uniqueness of a solution $y \colon [0,T] \to (0,+\infty)$ to the integro-differential equation
\begin{align}
\label{eq:IDE}
\begin{dcases}
\dfrac{\dd}{\dd s} y_s = -\delta y_s + \int_s^T \de^{-(\rho+\delta)(u-s)} y_u^{-\beta} \, \dd u, &s \in [0,T], \\
y_0 = x \, .
\end{dcases}
\end{align}
for any initial datum $x > 0$.

Indeed, this equivalence follows from the next {definition} and from {Lemma}~\ref{prop:equiv_mild_classica} below, whose proof is omitted, being a standard result{, see, e.g.,~\citep[Chapter~1]{lakshmikantham1995:integraleq}.}

\begin{definition}\label{def:sol}
\mbox{}
\begin{enumerate}[label = (\roman*)]
    \item\label{def:sol:classica} A function $y\in \dC^1([0,T];(0,+\infty))$ that satisfies~\eqref{eq:IDE} for all $s \in [0,T]$ is called a \emph{classical solution} to~\eqref{eq:IDE}.
    \item\label{def:sol:mild} A function $y\in \dC([0,T];(0,+\infty))$ that satifies the integral equation
\begin{equation*}
y_s =\de^{-\delta s}x +\int_0^s \de^{-\delta (s-r)} \int_r^T \de^{-(\rho+\delta)(u-r)} y_u^{-\beta} \, \dd u \, \dd r,  \quad s \in [0,T],
\end{equation*}
is called a \emph{mild solution} to~\eqref{eq:IDE}.
\end{enumerate}
\end{definition}

\begin{lemma}\label{prop:equiv_mild_classica}
Any classical solution $y\in \dC^1([0,T];(0,+\infty))$ to~\eqref{eq:IDE} is also a mild solution. Vice versa, any mild solution $y\in \dC([0,T];(0,+\infty))$ to~\eqref{eq:IDE} is continuously differentiable in $[0,T]$ and also a classical solution.
\end{lemma}

\begin{remark}
From now on, we will just speak of a solution (in either sense) to~\eqref{eq:IDE}, unless otherwise specified.
\end{remark}
%

%

It is convenient to introduce a further integro-differential equation, namely
\begin{equation}
\label{eq:zIDE}
\begin{dcases}
\dfrac{\dd}{\dd s} z_s = \de^{(\rho+2\delta)s} \int_s^T \de^{-(\rho+\delta-\delta\beta)u} z_u^{-\beta} \, \dd u, &s \in [0,T], \\
z_0 = x\, ,
\end{dcases}
\end{equation}
and the corresponding integral equation
\begin{equation}
\label{eq:zIDEint}
z_s = x + \int_0^s \de^{(\rho+2\delta)r} \int_r^T \de^{-(\rho+\delta-\delta\beta)u} z_u^{-\beta} \, \dd u \, \dd r, \quad s \in [0,T] \, ,
\end{equation}
which we consider for any initial datum $x > 0$.

The notions of classical and mild solution to~\eqref{eq:zIDE} are analogous to those given in {Definition~\ref{def:sol}.} Also in this case, a result similar to {Lemma}~\ref{prop:equiv_mild_classica} shows that these two solution concepts are equivalent. 

The following result, whose proof is omitted being based on routine computations,  shows that solutions to~\eqref{eq:IDE} can be expressed in terms of those of~\eqref{eq:zIDE} and vice versa. Thus, we can study either of these two equations, depending on convenience.
\begin{lemma}\label{prop:IDE_transformation}
If $y \colon [0,T] \to (0,+\infty)$ is a solution to~\eqref{eq:IDE}, then $z_s \coloneqq \de^{\delta s} y_s$, $s \in [0,T]$ is a solution to~\eqref{eq:zIDE}.
Vice versa, if $z \colon [0,T] \to (0,+\infty)$ is a solution to~\eqref{eq:zIDE}, then $y_s \coloneqq \de^{-\delta s} z_s$, $s \in [0,T]$ is a solution to~\eqref{eq:IDE}.
\end{lemma}

\begin{remark}
Note that equations~\eqref{eq:IDE} and~\eqref{eq:zIDE} do not fall into the standard theory of integro-differential equations (see, e.g.~\citep{gripenberg1990:integraleq, lakshmikantham1995:integraleq}), as we have an initial-value problem with the integral being backward in time.
\end{remark}

Now, we proceed to study existence and uniqueness of a solution to~\eqref{eq:IDE} or, equivalently, \eqref{eq:zIDE}. As in the previous section, we are going to divide our analysis into two subsections, the first devoted to the finite time horizon case, the second one to the infinite time horizon case.
{In particular, to prove the well-posedness of the integro-differential equation~\eqref{eq:IDE} in the finite time horizon case, we are going to rely on a fixed point argument to show existence of a solution, and on an \textit{ad-hoc} argument based on the structure of the equation to show its uniqueness. This result will be used in the infinite time horizon case to show existence of a solution to~\eqref{eq:IDE}, by constructing a candidate one as a sequence of solutions to the integro-differential equation with finite time horizon. Then, we prove uniqueness of the solution and its convergence to a steady state using arguments that rely on the analysis of its properties.}

\subsection{The finite time horizon case}
\label{sec:fixed_point:finite}

Let us fix $T < +\infty$ and $x > 0$. 
We observe, first, that solutions of \eqref{eq:zIDE} satisfy suitable a priori estimates.
\begin{proposition}
\label{pr:boundz}
Let $z\in \dC([0,T];(0,+\infty))$ be a solution to~\eqref{eq:zIDE}. Then, for all $s \in [0,T]$,
\begin{equation}
\label{eq:alphastima}
z_s \geq \underline \alpha^{x}_s , \qquad and \qquad
z_s\leq \overline \alpha^{x}_s,
\end{equation}
where
\begin{equation}
\label{eq:alpha_lowerbar}
\underline \alpha^{x}_s \coloneqq x , \quad s \in [0,T] \, ,
\end{equation}
and, for any $s \in [0,T]$,
\begin{equation}
\label{eq:alpha_upperbar}
\overline \alpha^{x}_s \coloneqq
\begin{dcases}
x - \dfrac{x^{-\beta}}{\rho+\delta-\delta\beta} \left\{\dfrac{1-\de^{\delta(1+\beta)s}}{\delta(1+\beta)}-\de^{-(\rho+\delta-\delta\beta)T}\dfrac{1-\de^{(\rho+2\delta)s}}{\rho+2\delta}\right\}, &\text{if } \beta \neq 1 + \dfrac \rho \delta, \\
x - \dfrac{x^{-\beta}}{(\rho+2\delta)^2}\left\{[1+(\rho+2\delta)(T-s)]\de^{(\rho+2\delta)s} - [1+(\rho+2\delta)T]\right\}, &\text{if } \beta = 1 + \dfrac \rho \delta.
\end{dcases}
\end{equation}
\end{proposition}
\begin{proof}
Since $z_s > 0$, for all $s \in [0,T]$,  the right hand side of~\eqref{eq:zIDEint} is positive, and hence, again for all $s \in [0,T]$,
\begin{equation}
\label{eq:alpha_lowerbarstima}
z_s \geq x, \qquad and \qquad 
z_s^{-\beta} \leq x^{-\beta},
\end{equation}
which provides us with the lower bound. To get the upper bound, it is enough to substitute the second inequality in~\eqref{eq:alpha_lowerbarstima} into the right hand side of~\eqref{eq:zIDEint} and compute the integrals.
\end{proof}

Given the result above our goal will be to show that~\eqref{eq:zIDE} admits a unique solution in the (nonempty) set
\begin{equation}
\label{eq:zIDE_solution_set}
\cC_{x} \coloneqq \{f \in \dC([0,T];(0,+\infty)) \colon \underline \alpha^{x}_s \leq f_s \leq \overline \alpha^{x}_s, \, \forall s \in [0,T]\}.
\end{equation}

\begin{theorem}
\label{th:fixed_point_finite}
For each fixed $x > 0$, there exists a unique solution to the integro-differential equation~\eqref{eq:zIDE} in the set $\cC_{x}$ defined in~\eqref{eq:zIDE_solution_set}, which is also twice continuously differentiable. Consequently, the integro-differential equation~\eqref{eq:IDE} admits
a unique solution $y\in
\dC^2([0,T];(0,+\infty))$  satisfying the bounds
\begin{equation}
\label{eq:IDE_solution_bounds}
\de^{-\delta s} x \leq y_s \leq \de^{-\delta s} \overline \alpha^{x}_s, \qquad s \in [0,T].
\end{equation}
\end{theorem}

\begin{proof}
Fix $x > 0$. Showing that~\eqref{eq:zIDE} admits a unique solution is equivalent to proving that the operator $\Gamma$, defined on $\dC([0,T];(0,+\infty))$ and given by
\begin{equation}
\label{eq:Gamma}
\Gamma(f)_s \coloneqq x + \int_0^s \de^{(\rho+2\delta)r} \int_r^T \de^{-(\rho+\delta-\delta\beta)u} f_u^{-\beta} \, \dd u \, \dd r, \quad s \in [0,T],
\end{equation}
admits a unique fixed point in $\cC_{x}$.

It is easy to verify that $\cC_{x}$ is a closed, bounded, and convex subset of $\dC([0,T];(0,+\infty))$ (which is endowed with the usual sup-norm). Moreover, $\Gamma$ maps $\cC_{x}$ into itself. Indeed, from~\eqref{eq:Gamma} we readily verify that, for all $f \in \cC_{x}$ and all $s \in [0,T]$, $\Gamma(f)_s \geq \underline{\alpha}^x_s$, since $f$ is a positive function; moreover, noting that for all $f,g \in \cC_{x}$
\begin{equation*}
f_s \geq g_s \quad \Longrightarrow \quad \Gamma(f)_s \leq \Gamma(g)_s, \quad s \in [0,T],
\end{equation*}
we immediately check that $\Gamma(f)_s \leq \Gamma(\underline {\alpha}^x)_s = \overline{\alpha}^x_s$.
Finally, $\Gamma$ is compact in $\cC_{x}$ as a consequence of the Ascoli-Arzelà theorem. Indeed, any sequence in $\cC_{x}$ is equi-bounded, thanks to the definition of $\cC_{x}$, and equi-continuous since, for all $f \in \cC_{x}$, $\Gamma(f)$ is differentiable and
\begin{equation*}
\sup_{f \in \cC} \sup_{s \in [0,T]} \abs{\Gamma(f)_s^\prime} = \sup_{f \in \cC} \sup_{s \in [0,T]} \left\{ \de^{(\rho+2\delta)s} \int_s^T \de^{-(\rho+\delta-\delta\beta)u} f_u^{-\beta} \, \dd u\right\} \leq \sup_{s \in [0,T]}  {\left|\frac{\dd }{\dd s}\overline{\alpha}^{x}_s\right|} < +\infty.
\end{equation*}
{In fact, since $f_s^{-\beta} \leq (\underline{\alpha}^x_s)^{-\beta} = x^{-\beta}$, for all $s \in [0,T]$, we have that, if $\beta \neq 1 + \frac\rho\delta$,
\begin{align*}
    0 
    &\leq \de^{(\rho+2\delta)s} \int_s^T \de^{-(\rho+\delta-\delta\beta)u} f_u^{-\beta} \, \dd u
    \\
    &\leq
    \dfrac{x^{-\beta}}{\rho+\delta-\delta\beta} \left\{\de^{\delta(1+\beta)s} - \de^{-(\rho+\delta-\delta\beta)T}\de^{(\rho+2\delta)s}\right\} = \frac{\dd }{\dd s}\overline{\alpha}^{x}_s, \quad s \in [0,T],
\end{align*}
while if $\beta = 1 + \frac\rho\delta$,
\begin{equation*}
    0 \leq \de^{(\rho+2\delta)s} \int_s^T \de^{-(\rho+\delta-\delta\beta)u} f_u^{-\beta} \, \dd u
    \leq
    x^{-\beta}(T-s)\de^{(\rho+2\delta)s} = \left|\frac{\dd }{\dd s}\overline{\alpha}^{x}_s\right|, \quad s \in [0,T].
\end{equation*}%
}%
Therefore, by Schauder's fixed point theorem the map $\Gamma$ has a fixed point $z \in \cC_{x}$, i.e., $z$ verifies~\eqref{eq:zIDEint}. Hence, $z$ is a solution of equation~\eqref{eq:zIDE}. 
The fact that $z \in \dC^2([0,T];(0,+\infty))$ simply follows by differentiating twice~\eqref{eq:zIDEint}.

To show uniqueness, assume that two different solutions $z, \widetilde z \in \cC_{x}$ of~\eqref{eq:zIDE} exist.
It is convenient to write the integro-differential equation satisfied by $w_s \coloneqq z_{T-s}$ and $\widetilde w_s \coloneqq \widetilde z_{T-s}$, $s \in [0,T]$,
\begin{equation}
\label{eq:wIDE_time_reverse}
\begin{dcases}
\dfrac{\dd}{\dd s} w_s = -\de^{\delta(1+\beta)(T-s)} \int_0^s \de^{-(\rho+\delta-\delta\beta)(s-r)} w_r^{-\beta} \, \dd r, &s \in [0,T], \\
w_{T} = x\, .
\end{dcases}
\end{equation}
From this equation we deduce that
\begin{equation}
\label{eq:wIDE_difference}
w^\prime_s - \widetilde w^\prime_s = -\de^{\delta(1+\beta)(T-s)} \int_0^s \de^{-(\rho+\delta-\delta\beta)(s-r)} \left[w_r^{-\beta} - \widetilde w_r^{-\beta}\right] \dd r, \quad s \in [0,T] \, .
\end{equation}
The following two cases may occur:

\smallskip

\noindent\textbf{Case 1.} Consider $w_0 \neq \widetilde w_0$ and assume, without loss of generality, that $w_0 > \widetilde w_0$. Let us define
\begin{equation*}
s^\star \coloneqq \inf\{u > 0 \colon w_u = \widetilde w_u\} \, .
\end{equation*}
For each fixed $s \in [0, s^\star)$, we have that $w_s > \widetilde w_s$ and, hence, $w_r^{-\beta} - \widetilde w_r^{-\beta} < 0$, for all $r \in [0,s]$; thus, we deduce from~\eqref{eq:wIDE_difference} that $w^\prime_s - \widetilde w^\prime_s > 0$, i.e., that the distance between the two solutions increases in time. Therefore, $s^\star = +\infty$, but this contradicts the fact that $w_{T} = \widetilde w_{T}$.

\smallskip

\noindent\textbf{Case 2.} Consider $w_0 = \widetilde w_0$ and note that both $w$ and $\widetilde w$ are bounded below by the constant $x$, since $z, \widetilde z \in \cC_{x}$. Observe, also, that the function $w \mapsto w^{-\beta}$ is Lipschitz in $[x, +\infty)$, with Lipschitz constant $L \coloneqq x^{-\beta-1}$. Therefore, integrating~\eqref{eq:wIDE_difference} on $[0,s]$ we get
\begin{equation*}
w_s - \widetilde w_s = -\int_0^s \de^{\delta(1+\beta)(T-u)} \int_0^u \de^{-(\rho+\delta-\delta\beta)(u-r)} \left[w_r^{-\beta} - \widetilde w_r^{-\beta}\right] \dd r \, \dd u, \quad s \in [0,T] \, ,
\end{equation*}
whence
\begin{equation*}
\abs{w_s - \widetilde w_s} \leq L \int_0^s \left\{\de^{\delta(1+\beta)(T-u)} \int_0^u \de^{-(\rho+\delta-\delta\beta)(u-r)} \, \dd r \right\} \sup_{r \in [0,u]} \abs{w_r- \widetilde w_r} \, \dd u, \quad s \in [0,T] \, .
\end{equation*}
Since $\displaystyle s \mapsto \sup_{r \in [0,s]} \abs{w_r - \widetilde w_r}$ is bounded on $[0,T]$, applying Gronwall's lemma we get that $w_s =\widetilde w_s$ for all $s \in [0, T]$, which contradicts the assumption that $w$ and $\widetilde w$ are different.

Therefore, equation~\eqref{eq:zIDE} admits a unique solution in $\cC_{x}$. 
Finally, from {Lemma}~\ref{prop:IDE_transformation}, we immediately deduce that also~\eqref{eq:IDE} has a unique solution in $\dC^2([0,T];(0+\infty))$ satisfying the bounds in~\eqref{eq:IDE_solution_bounds}.
\end{proof}

Thanks to the result above, we can establish the following.

\begin{theorem}
\label{prop:MF_exist_uniq_finite}
For each fixed random variable $\xi$ satisfying Assumption~\ref{hp:xi}, there exists a unique equilibrium $(\widehat \bfu, \widehat{q})$ for the mean-field game with finite time horizon, among all equilibria {$(\widetilde \bfu, \widetilde q) \in \cU_T \times \cQ_T$.}

More precisely, $\widehat \bfu = z^{(T,\widehat{q})}$, where $z^{(T,\widehat{q})}$ is the function defined in~\eqref{eq:z_finite}, and $\widehat{q}$ is the unique solution to~\eqref{eq:IDE}, with initial condition $x = \bbE[\xi]$.
\end{theorem}

\begin{proof}
Let $\widehat{q}$ be the unique solution to~\eqref{eq:IDE}, with $x = \bbE[\xi]$. Since $\widehat{q}$ satisfies the bounds~\eqref{eq:IDE_solution_bounds} and since the functions $s \mapsto \de^{-\delta s}x$ and $s \mapsto \de^{-\delta s} \overline \alpha_s^x$ are continuous on $[0,T]$, hence therein bounded, we deduce that also $\widehat{q}$ is bounded on $[0,T]$. More precisely,
\begin{equation}
0 < m \leq \widehat{q}_s \leq M, \quad \forall s \in [0,T], \quad \text{with } m \coloneqq \de^{-\delta T}x \, \text{ and } M \coloneqq \sup_{s \in [0,T]} \de^{-\delta s} \overline \alpha_s^x \, .
\end{equation}

Therefore, {$\widehat{q} \in \cQ_T$} and applying Theorem~\ref{th:MF_IDE}-\ref{th:MF_IDE_finite_onlyif} we get the result. 
\end{proof}

\subsection{The infinite time horizon case}
\label{sec:fixed_point:infinite}
We discuss now the infinite time horizon case, i.e., we fix $T = +\infty$.
Here we work under the assumption $\beta < 1 + \frac \rho \delta$,
which will be used to guarantee convergence of the integral appearing in~\eqref{eq:zIDE}.

%
A reasoning similar to the one used to prove Proposition~\ref{pr:boundz} shows that the following holds.
\begin{proposition}
\label{pr:boundzinfty}
Let $z\in \dC([0,+\infty);(0,+\infty))$ be a solution to~\eqref{eq:zIDE}, with $T = +\infty$, and assume that $\beta < 1 + \frac \rho \delta$. Then, for all $s \geq 0$,
\begin{equation}
\label{eq:alphastimainfty}
z_s \geq \underline \alpha^{x, \infty}_s , \qquad and \qquad
z_s\leq \overline \alpha^{x, \infty}_s.
\end{equation}
where
\begin{equation}
\label{eq:alpha_lowerbarinfty}
\underline \alpha^{x, \infty}_s \coloneqq x , \quad s \geq 0 \, ,    
\end{equation}
and
\begin{equation}
\label{eq:alpha_upperbar_infinite}
\overline \alpha^{x,\infty}_s \coloneqq
x - \dfrac{x^{-\beta}}{\rho+\delta-\delta\beta} \left\{\dfrac{1-\de^{\delta(1+\beta) s}}{\delta(1+\beta)}\right\}, \quad s \geq 0.
\end{equation}
\end{proposition}

{As the notation suggests, the function $\overline \alpha^{x,\infty}$ given in~\eqref{eq:alpha_upperbar_infinite} is the pointwise limit, as $T \to +\infty$, of the sequence $\{\overline \alpha^{x,T}\}_{T > 0}$, where for each $T > 0$ the function $\overline \alpha^{x,T}$ has the expression given in~\eqref{eq:alpha_upperbar}.}

We aim at showing that~\eqref{eq:zIDE} admits a unique solution in the (nonempty) set
\begin{equation}
\label{eq:zIDE_solution_set_infinite}
\cC_{x,\infty} \coloneqq \{f \in \dC([0,+\infty);(0,+\infty)) \colon \underline \alpha^{x, \infty}_s \leq f_s \leq \overline \alpha^{x,\infty}_s, \, \forall s \geq 0\}.
\end{equation}
Note, in particular, that for any $f \in \cC_{x,\infty}$ we have that
\begin{equation*}
\int_s^{+\infty} \de^{-(\rho+\delta-\delta\beta)u} f_u^{-\beta} \, \dd u \leq
\int_s^{+\infty} \de^{-(\rho+\delta-\delta\beta)u} x^{-\beta} \, \dd u 
= \frac{x^{-\beta}}{\rho+\delta-\delta\beta} \de^{-(\rho+\delta-\delta\beta)s}
<+\infty, \quad \forall s \geq 0,
\end{equation*}
as $f$ is bounded below by the constant $x$ and we assumed $\beta < 1 + \frac \rho \delta$.

\begin{theorem}
\label{th:fixed_point_infinite}
Let $T = +\infty$ and assume that $\beta < 1 + \frac \rho \delta$. Then, for each fixed $x > 0$, there exists a unique solution to the integro-differential equation~\eqref{eq:zIDE} in the set $\cC_{x,\infty}$ defined in~\eqref{eq:zIDE_solution_set_infinite}, which is also twice continuously differentiable. Moreover, the integro-differential equation~\eqref{eq:IDE} admits
a unique solution $y\in \dC^2([0,+\infty);(0,+\infty))$  satisfying the bounds
\begin{equation}
\label{eq:IDE_solution_bounds_infinite}
\de^{-\delta s} x \leq y_s \leq \de^{-\delta s} \left[x - \dfrac{x^{-\beta}}{\delta(1+\beta)(\rho+\delta-\delta\beta)} \right] + \dfrac{\de^{\delta\beta s} x^{-\beta}}{\delta(1+\beta)(\rho+\delta-\delta\beta)}, \quad s \geq 0.
\end{equation}
Finally, if $x = y_\infty \coloneqq (\delta(\rho+\delta))^{-\frac{1}{\beta+1}}$, then the function $y_s = y_\infty$, for all $s \geq 0$, is the unique constant solution to~\eqref{eq:IDE}. 
If, instead, $x \neq y_\infty$, then the corresponding solution $y$ to~\eqref{eq:IDE} is strictly monotone and such that
\begin{equation*}
\lim_{s \to +\infty} y_s = y_\infty \, .
\end{equation*}
\end{theorem}

\begin{proof}
We divide the proof as follows. First, we show existence and uniqueness of a solution to~\eqref{eq:zIDE} in the set $\cC_{x,\infty}$, for each fixed $x > 0$. 
As in the finite time horizon case, the fact that such solution is twice continuously differentiable follows by differentiation of~\eqref{eq:zIDEint}.
Then, it immediately follows from {Lemma}~\ref{prop:IDE_transformation} that~\eqref{eq:IDE} admits a unique solution $y \in \dC^2([0,+\infty);(0,+\infty))$ satisfying the bounds given in~\eqref{eq:IDE_solution_bounds_infinite}.
Finally, we will show that the constant function equal to $y_{\infty}$ is the unique constant solution to~\eqref{eq:IDE}, 
and that the solutions starting at $x \neq y_\infty$ are monotone in time and converge to $y_\infty$.

\medskip

\textit{Existence.}
Fix $x >0$, and consider the family of solutions $\cZ \coloneqq \{z^{(n)}\}_{n \geq 1}$, where $z^{(n)}$ is the solution to the finite horizon problems with terminal time $T = n$, i.e.,
\begin{equation}\label{eq:zIDE_n}
\begin{dcases}
\dfrac{\dd}{\dd s} z^{(n)}_s = \de^{(\rho+2\delta)s} \int_s^n \de^{-(\rho+\delta-\delta\beta)u} \bigl(z_u^{(n)}\bigr)^{-\beta} \, \dd u, &s \in [0, n], \\
z_0 = x\, .
\end{dcases}
\end{equation}
We extend each of them by continuity taking them to be constant over $(n,+\infty)$, so that $z^{(n)} \in \cC_{x,\infty}$, as $\overline \alpha_s^{x} \leq \overline \alpha_s^{x,\infty}$, for all $s \in [0,n]$.
Differently from the finite horizon case, here we cannot directly apply the Ascoli-Arzelà theorem, as the infinite horizon does not allow to get the same equicontinuity estimates as in the proof of Theorem~\ref{th:fixed_point_finite}. Hence, use a weaker compactness criterion based on the Fréchet-Kolmogorov theorem.

Let us define the set
\begin{equation*}
\dL^{1}_{\rho+2\delta}([0,\infty);\R) \coloneqq \left\{f \colon [0,\infty)\to \R \, \colon \, \int_{0}^{\infty}\de^{-(\rho+2\delta)s} \abs{f_s} \, \dd s < +\infty \right\} \, ,
\end{equation*}
endowed with the usual weighted norm $\norm{\cdot}_{1,\rho+2\delta}$.
Clearly $\cZ \subset \dL^{1}_{\rho+2\delta}([0,\infty);\R)$.

We observe that, given the bounds in~\eqref{eq:zIDE_solution_set_infinite}, the family $\cZ$ is equibounded. We want to show that it is also equi-integrable in $\dL^{1}_{\rho+2\delta}$. For any $s \geq 0$ and $h > 0$ we have that
\begin{equation*}
z^{(n)}_{s+h} - z^{(n)}_s = \ind_{[0,n-h]}(s) \int_s^{s+h} \bigl(z^{(n)}_r\bigr)^\prime \, \dd r + \ind_{[n-h,n]}(s) \int_s^n \bigl(z^{(n)}_r\bigr)^\prime \, \dd r.
\end{equation*}
Let $K \coloneqq \frac{x^{-\beta}}{\rho+\delta-\delta\beta}$. Using the fact that $0 < \bigl(z^{(n)}_s\bigr)^\prime \leq K\de^{\delta(1+\beta)s}$, for all $s \geq 0$, and all $n \geq 1$, we get that, for any $h > 0$ small enough,
\begin{align*}
\allowdisplaybreaks
&\phantom{\leq} \norm{z^{(n)}_{\cdot+h} - z^{(n)}_{\cdot}}_{1,\rho+2\delta} = \int_0^{+\infty} \de^{-(\rho+2\delta) s} \abs{z^{(n)}_{s+h} - z^{(n)}_s} \, \dd s
\\
&\leq \int_{0}^{n-h} \de^{-(\rho+2\delta) s} \int_s^{s+h} \bigl(z^{(n)}_r\bigr)^\prime \, \dd r \, \dd s + \int_{n-h}^n \de^{-(\rho+2\delta) s} \int_s^n \bigl(z^{(n)}_r\bigr)^\prime \, \dd r
\\
&\leq \dfrac{K(\de^{\delta(1+\beta)h}-1)}{\delta(1+\beta)} \int_{0}^{n-h} \de^{-(\rho+\delta-\delta\beta) s} \, \dd s
\\
&\qquad + \dfrac{K \de^{\delta(1+\beta)n}}{\delta(1+\beta)} \int_{n-h}^n \de^{-(\rho+2\delta)s} \, \dd s \underbrace{- \dfrac{K}{\delta(1+\beta)} \int_{n-h}^n \de^{-(\rho+\delta-\delta\beta)s} \, \dd s}_{\leq 0}
\\
&\leq \dfrac{K(\de^{\delta(1+\beta)h}-1)}{\delta(1+\beta)(\rho+\delta-\delta\beta)} [1 \underbrace{- \de^{-(\rho+\delta-\delta\beta) (n-h)}}_{\leq 0}] + \dfrac{K \de^{-(\rho+\delta-\delta\beta)n}}{\delta(1+\beta)(\rho+2\delta)}[\de^{(\rho+2\delta)h}-1]
\\
&\leq \dfrac{K(\de^{\delta(1+\beta)h}-1)}{\delta(1+\beta)(\rho+\delta-\delta\beta)} + \dfrac{K}{\delta(1+\beta)(\rho+2\delta)}[\de^{(\rho+2\delta)h}-1], \quad \forall n \geq 1 \, ,
\end{align*}
{where in the last inequality we used the condition $\beta < 1+\frac\rho\delta$.}
Therefore, we obtain that
\begin{equation*}
\lim_{h\to 0^{+}} \sup_{n \geq 1} \norm{z^{(n)}_{\cdot+h} - z^{(n)}_{\cdot}}_{1,\rho+2\delta} = 0,
\end{equation*}
i.e., that the family $\cZ$ is equi-integrable in
$\dL^{1}_{\rho+2\delta}$.

By the Frechet-Kolmogorov theorem, it follows that $\cZ$ is relatively compact, so we can extract a subsequence $\{z^{(n)}\}_{n \geq 1}\subset \cZ$, still labeled (with an abuse of notation) by $\{z^{(n)}\}_{n \geq 1}$, such that $z^{(n)} \longrightarrow z^{(\infty)}$
with respect to $\norm{\cdot}_{1,\rho+2\delta}$, as $n \to \infty$. Then, we can extract a sub-subsequence, still labeled (with an abuse of notation) by $\{z^{(n)}\}_{n \geq 1}$, such that $z^{(n)}_s \longrightarrow z^{(\infty)}_s$, for almost every $s \geq 0$, as $n \to \infty$. 
Clearly, the limit $z^{(\infty)}$ belongs to $\dL^{1}_{\rho+2\delta}$ and, since $z^{(n)} \in \cC_{x,\infty}$, for any $n \geq 1$, we have that $\underline{\alpha}_s^{x,\infty} \leq z^{(\infty)}_s \leq \overline \alpha_s^{x,\infty}$, for almost every $s \geq 0$, where $\underline\alpha^{x,\infty}$ and $\overline\alpha^{x,\infty}$ are the functions defined in~\eqref{eq:alpha_lowerbarinfty} and~\eqref{eq:alpha_upperbar_infinite}. Now we use \eqref{eq:zIDE} to get that also the derivative 
$\{\left(z^{(n)}\right)^\prime\}$ of the last subsequence $\{z^{(n)}\}$ converges a.e. to the function
\begin{equation*}
w^{(\infty)}_s \coloneqq \de^{(\rho+2\delta)s}\int_s^{+\infty} \de^{-(\rho+\delta-\delta\beta)u}\left(z^{(\infty)}_u\right)^{-\beta} \, \dd u, \quad s \geq 0 \, ,
\end{equation*}
which belongs to $\dL^{1}_{\rho+2\delta}$. By standard computations, it follows that $\{\left(z^{(n)}\right)^\prime\} \subset \dL^1_{\rho+2\delta}$ and that 
\begin{equation*}
0 < \left(z^{(n)}_s\right)^\prime \leq \dfrac{x^{-\beta}}{\rho+\delta-\delta\beta} \de^{\delta(1+\beta)s} \eqqcolon g_s, \quad \forall s \geq 0, \, \forall n \geq 1,
\end{equation*}
with $g \in \dL^1_{\rho+2\delta}$. Therefore, the sequence of  derivatives 
$\{\left(z^{(n)}\right)^\prime\}$ 
converges in $\dL^1_{\rho+2\delta}$ to $w^{(\infty)}$. 
We deduce that the  subsequence $\{z^{(n)}\}$ converges in the weighted Sobolev space $W^{1,1}_{\rho+2\delta}$ (cf.~\citep[Chapter~8, Remark~4]{brezis:functionalanalysis}). By completeness of this space,
we get that $w^{(\infty)}_s=(z^{(\infty)}_s)^\prime$, for a.e. $s \geq 0$. Moreover, using the same argument as in the proof of~\citep[Theorem~8.2]{brezis:functionalanalysis}, we get that there exists a function $\overline z^{(\infty)} \in \dC([0,+\infty);(0,+\infty))$ such that $\overline z^{(\infty)}_s = z^{(\infty)}_s$, for almost every $s \geq 0$, and 
\begin{equation*}
    \overline z^{(\infty)}_s = x + \int_0^s (z^{(\infty)}_u)^\prime \, \dd u = x +\int_0^s w^{(\infty)}_u \, \dd u, \quad \forall s \geq 0 \, .
\end{equation*}
Substituting the expression of $w^{(\infty)}$ in the previous equality we get that
\begin{equation*}
    \overline z^{(\infty)}_s = x + \int_0^s \de^{(\rho+2\delta)r}\int_r^{+\infty} \de^{-(\rho+\delta-\delta\beta)u}\left(\overline z^{(\infty)}_u\right)^{-\beta} \, \dd u \, \dd r \quad s \geq 0,
\end{equation*}
that is, $\overline z^{(\infty)} \in \cC_{x,\infty}$ and solves~\eqref{eq:zIDE}.

\medskip

\textit{Uniqueness.}
As we observed at the beginning of the proof, any solution $z$ to~\eqref{eq:zIDE} is (at least) twice differentiable.
Differentiating~\eqref{eq:zIDE} with respect to the time variable and denoting by $'$ and $''$ first- and second-order derivatives with respect to this variable, we have that $z$ verifies
\begin{equation}
\label{eq:zIDE_infinitebis}
\begin{dcases}
z''_s = (\rho+2\delta)z'_{s}- \de^{\delta (1+\beta)s}z_{s}^{-\beta}, &s \geq 0, \\
z_0 = x \, .
\end{dcases}
\end{equation}
The equation above is a second-order ODE with locally Lipschitz coefficients over $(s,z)\in [0,+\infty)\times (0,+\infty)$, to which we can associate the following initial value problem
\begin{equation}
\label{eq:zIDE_infinitetris}
\begin{dcases}
z''_s = (\rho+2\delta)z'_{s}- \de^{\delta (1+\beta)s}z_{s}^{-\beta}, &s \geq 0, \\
z_0 = x \, , \\
z'_0 =\zeta > 0 \, .
\end{dcases}
\end{equation}
Note that we only need to consider $\zeta > 0$, since solutions to~\eqref{eq:zIDE} are such that their first derivative is strictly positive for all $s \geq 0$.
For each fixed $x > 0$ and $\zeta > 0$, \eqref{eq:zIDE_infinitetris} has a unique solution $z^{x,\zeta}$ on the maximal interval $[0,\tau^{*}(x,\zeta))$, with $\tau^{*}(x,\zeta) \leq +\infty$.

Since any solution $z$ to~\eqref{eq:zIDE} also verifies~\eqref{eq:zIDE_infinitebis}, it also satisfies~\eqref{eq:zIDE_infinitetris} for some $\overline \zeta > 0$. Therefore, if we show that for each given $x > 0$, there exists a unique $\zeta > 0$ such that~\eqref{eq:zIDE_infinitetris} has a unique global solution $z^{x,\zeta} \in \mathcal{C}_{x,\infty}$, then necessarily $\zeta= \overline \zeta$ and $z^{x,\zeta}$ must also be the unique solution to~\eqref{eq:zIDE}.
The idea is, thus, to study, for each fixed $x > 0$, the dependence of the solution $z^{x,\zeta}$ to~\eqref{eq:zIDE_infinitetris} on $\zeta > 0$.

%

Let us fix $x > 0$. Standard results (see, e.g.,~\citep[Theorem~7,5, Chapter~1]{coddingtonlevinson1955:ode}) ensure that the solution $z^{x,\zeta}$ to~\eqref{eq:zIDE_infinitetris} and its first- and second-order derivatives with respect to the time variable are differentiable with respect to $\zeta$ and that the order of differentiation can be exchanged.

Therefore, defining $\widehat z_s({\zeta}) \coloneqq \frac{\partial}{\partial \zeta} z_{s}^{x,\zeta}$, for all $s \geq 0$, we get from~\eqref{eq:zIDE_infinitetris} that $\widehat z(\zeta)$ solves the initial value problem
\begin{equation}
\label{eq:hatzIDE}
\begin{dcases}
\widehat z''_s(\zeta) = (\rho+2\delta) \widehat z'_{s}(\zeta)+\beta \de^{\delta (1+\beta)s}z_{s}(\zeta)^{-\beta-1}\widehat z_{s}(\zeta), &s \geq 0, \\
\widehat z_0 (\zeta)= 0 \, , \\
\widehat z'_0(\zeta) = 1 \, ,
\end{dcases}
\end{equation}
where $z_s(\zeta) \coloneqq z_s^{x,\zeta}$.
Let 
\begin{equation*}
s^{*}(\zeta) \coloneqq \inf\{s\geq 0 \colon \widehat z'_{s}(\zeta)\leq 0\}>0.
\end{equation*}
By definition $z'_{s}(\zeta)>0$, for all $s\in[0,s^{*}(\zeta))$, which implies
\begin{equation*}
\widehat z_{s}(\zeta)>0, \qquad \forall s\in[0,s^{*}(\zeta)).
\end{equation*}
Therefore, using~\eqref{eq:hatzIDE} we get that $\widehat z''_{s}(\zeta)>0$, for all $s\in[0,s^{*}(\zeta))$.
It follows that
\begin{equation*}
\widehat z'_{s}(\zeta)\geq 1, \qquad \forall s\in[0,s^{*}(\zeta)),
\end{equation*}
which implies that $s^{*}(\zeta)=+\infty$ and, in turn,  that $\widehat z_{s}(\zeta)\geq s$, for all $s\geq 0$.

Using again~\eqref{eq:hatzIDE}, 
we then get
\begin{equation}
\label{eq:hatzIDEnew}
\begin{dcases}
\widehat z''_s(\zeta) \ge (\rho+2\delta) \widehat z'_{s}(\zeta), &s \geq 0, \\
\widehat z'_0(\zeta) = 1 \, ,
\end{dcases}
\end{equation}
which, implies, by comparison,
$$
\widehat z'_{s}(\zeta)\ge \de^{(\rho+2\delta) s}, \qquad \forall s\geq 0,
$$
and, consequently,
%
\begin{equation*}
\widehat z_{s}(\zeta)\geq 	 \frac{1}{\rho+2\delta}[\de^{(\rho+2\delta)s}-1], \qquad \forall s\geq 0.
\end{equation*}

Thus, for every $a<b$, we have

\begin{equation*}
z_{s}(b)-z_s(a)\geq  \int_{a}^{b}\widehat z_{s}(\zeta) \, \dd \zeta\geq (b-a)  \frac{1}{\rho+2\delta}
[\de^{(\rho+2\delta)s}-1],  \qquad \forall s\geq 0.
\end{equation*}

We know that, for some $\overline \zeta > 0$, there exists a global solution $z(\overline \zeta) \in \mathcal{C}_{x,\infty}$ to~\eqref{eq:zIDE_infinitetris}, since any solution $z$ to~\eqref{eq:zIDE} also verifies this equation. Using the estimate above and~\eqref{eq:zIDE_solution_set_infinite}, we get, for all $b > a \coloneqq \overline \zeta$,

\begin{equation*}
z_{s}(b) \geq z_s(\overline \zeta) + (b-\overline \zeta) {\frac{1}{\rho+2\delta}} [\de^{(\rho+2\delta)s}-1] \geq x +  (b-\overline \zeta) \frac{1}{\rho+2\delta}
[\de^{(\rho+2\delta)s}-1], \qquad \forall s\geq 0.
\end{equation*}

This implies that $z(b)$ would grow exponentially at a rate at least $\rho+2\delta > \delta(1+\beta)$, which contradicts the fact that solutions in~\eqref{eq:zIDE_solution_set_infinite} need to grow at a rate at most $\delta(1+\beta)$.

Similarly, we get, for all $a  < \overline \zeta \eqqcolon b$,

\begin{equation*}
z_{s}(a) \leq z_s(\overline \zeta) - (\overline \zeta-a) {\frac{1}{\rho+2\delta}} [\de^{(\rho+2\delta)s}-1] \leq \overline \alpha_s^{x,\infty} - (\overline \zeta-a) \frac{1}{\rho+2\delta}
[\de^{(\rho+2\delta)s}-1], \qquad \forall s\geq 0.
\end{equation*}

This implies that $z(a)$ becomes negative in finite time, which contradicts the fact that solutions in~\eqref{eq:zIDE_solution_set_infinite} are positive.

By arbitrariness of $b > \overline \zeta$ and $a < \overline \zeta$, we conclude.

\medskip

\textit{Convergence.} 
As recalled at the beginning of the proof, uniqueness of the solution to~\eqref{eq:zIDE} implies uniqueness of the solution to~\eqref{eq:IDE}, which is also twice continuously differentiable.
Differentiating~\eqref{eq:IDE}, we get the second-order ODE
\begin{equation}\label{ODEy}
y''_s = \rho y'_{s}+ (\rho+\delta)\delta y_s - y_s^{-\beta}, \quad s \geq 0.
\end{equation}
Hence, any solution to \eqref{eq:IDE} is also a solution to~\eqref{ODEy}.
We observe that $y_\infty$ is the unique solution to the algebraic equation
\begin{equation}\label{eq:segno}
 (\rho+\delta)\delta y - y^{-\beta}=0.       
\end{equation}
Hence,  the function $y_s = y_{\infty}$, for all $s \geq 0$, is the unique stationary solution to~\eqref{eq:IDE}.

Now, we show monotonicity and convergence of the solution to~\eqref{eq:IDE} when $x \neq y_\infty$. We prove the case $x>y_{\infty}$, as the other one can be established by similar arguments.

Let $y$ be the unique solution to~\eqref{eq:IDE} and fix $x > y_\infty$. We prove, first, that $y_s>y_\infty$, for all $s \geq 0$. Let us define
\begin{equation*}
s_{0} \coloneqq \inf\{s\geq 0 \colon y_{s}=y_{\infty}\} \, ,
\end{equation*}
and assume, by contradiction, that  $s_{0}<+\infty$.
We have two cases:
\begin{itemize}
\item[(i)] $y_{s_{0}}'<0$;
\item[(ii)] $y_{s_{0}}'=0$.
\end{itemize}
Consider the first case. By the flow property, the function $\widetilde y_s \coloneqq y_{s-s_0}$, $s \geq s_0$, solves~\eqref{eq:IDE} with initial condition equal to $y_\infty$ and it satisfies the a priori bounds given in~\eqref{eq:IDE_solution_bounds_infinite} (with $x = y_\infty$). However, this solution is different from the constant one, which is the unique solution to~\eqref{eq:IDE} with initial condition $y_\infty$ and satisfying the aforementioned bounds. Hence we have a contradiction.

The second case leads to a contradiction as well. Indeed, we would have two different solutions in the interval $[0,s_{0}]$ to the Cauchy problem~\eqref{ODEy}, with terminal conditions $y_{s_{0}}=y_{\infty}$ and $y'_{s_{0}}=0$, which clearly admits unique solution, that is, $y_s = y_{\infty}$, for all $s \geq 0$. 

Hence, we have proved that $y_s > y_{\infty}$, for all $s \geq 0$. This implies, by strict monotonicity of the map $y \mapsto (\rho+\delta)\delta y - y^{-\beta}$ and \eqref{eq:segno}, that
\begin{equation}\label{eq:y_monotonia}
(\rho+\delta)\delta y_s - y_s^{-\beta} > (\rho+\delta)\delta y_{\infty} - y_{\infty}^{-\beta} = 0, \quad \forall s \geq 0 \, ,
\end{equation}
and hence, from~\eqref{ODEy}, that
\begin{equation}\label{eq:derivataseconda}
y_s'' > \rho y'_s, \quad \forall s \geq 0 .
\end{equation}
Now we show that $y'_s<0$, for all $s \geq 0$. Let
\begin{equation*}
s_{1} \coloneqq \inf\{s\geq 0 \colon y'_{s}\geq 0\}.
\end{equation*}
Assume, by contradiction, that $s_1 < +\infty$.
Then, from \eqref{eq:derivataseconda} we get that $y''_{s_1}> 0$.
It follows that there exist $\eta,\varepsilon>0$ such that $y'_{s_1+\varepsilon} = \eta > 0$. Using \eqref{eq:derivataseconda} again, we get
\begin{equation}
\label{eq:sisnewdersec}
\begin{dcases}
y''_{s} > \rho y'_{s} , & \forall s \geq s_1+\epsilon,\\
y'_{s_1+\varepsilon}  = \eta>0.
\end{dcases}
\end{equation}
Therefore, solving the corresponding Cauchy problem and using the comparison, we get
\begin{equation}
\label{eq:contradictionarg}
y'_s\geq  \eta \de^{\rho [s-(s_1+\varepsilon)]} > 0, \quad s\geq s_1+\varepsilon.
\end{equation}
Next, differentiating~\eqref{ODEy}, we get that $y \in \dC^3([0,+\infty); (0,+\infty))$ and that it satisfies
\begin{equation}\label{ODEybis}
y'''_{s}=\rho y''_{s}+(\rho+\delta)\delta y'_{s}+\beta y_{s}^{-\beta-1}y'_{s}, \quad s \geq 0.
\end{equation}
Since, by~\eqref{eq:contradictionarg}, we know that $y_s' >0$ for all $s \geq s_1+\varepsilon$, we deduce that 
\begin{equation*}
\begin{dcases}
    y_s''' > \rho y_s''+(\rho+\delta)\delta y'_{s}, & \forall s \geq s_1 +\varepsilon, \\
    y'_{s_1+\varepsilon}=\eta, \\
    y''_{s_1+\varepsilon} = \kappa,
\end{dcases}
\end{equation*}
with $\eta > 0$ and $\kappa > \eta\rho > 0$, by~\eqref{eq:sisnewdersec}.
Solving the corresponding Cauchy problem and using the comparison, we get
\begin{equation*}
    y'_s > \dfrac{\eta\delta + \kappa}{\rho+2\delta} \de^{(\rho + \delta)[s-(s_1+\varepsilon)]}
+\dfrac{\eta(\rho+\delta)-\kappa}{\rho+2\delta} \de^{-\delta[s-(s_1+\varepsilon)]}, \quad \forall s \geq s_1 +\varepsilon \, .
\end{equation*}
This implies that $y$ grows, in the long run, at least with rate  $\rho+\delta$ which is strictly bigger than $\delta\beta$, as prescribed by the admissibility condition in 
\eqref{eq:IDE_solution_bounds_infinite}. The contradiction follows, and hence we proved that $y'_s<0$, for all $s \geq 0$.

%
%

We are now going to show that $y''_s>0$, for all $s \geq 0$. 
Since we proved that $y_s' <0$, for all $s \geq 0$, we deduce from~\eqref{ODEybis} that 
\begin{equation}\label{eq:derivataterza}
    y_s''' < \rho y_s'', \quad \forall s \geq 0 \, .
\end{equation}
Assume by contradiction that $y''_{s_{2}}\leq 0$ for some $s_{2} \geq 0$. Then, from~\eqref{eq:derivataterza} we get that $y'''_{s_2} < 0$.
It follows that there exist $\theta<0$, $\tau >0$ such that $y''_{s_2+\tau} = \theta < 0$. Using again~\eqref{eq:derivataterza}, we get
\begin{equation*}
\begin{dcases}
y'''_{s} < \rho y''_{s}, & \forall s \geq s_2 + \tau, \\
y''_{s_2+\tau} = \theta.
\end{dcases}
\end{equation*}
Therefore, $y''_{s}\leq \theta \de^{\rho s} < 0$, for all $s \geq s_2 + \tau$. Considering that $y_{s}'<0$, for all $s \geq 0$, this implies that the graph of $s \mapsto y_s$ lies below a straight line with strictly negative slope, contradicting the fact that $y_s >y_{\infty}$, for all $s \geq 0$, as previously established.

So, we have proved that, for all $s \geq 0$, 
\begin{equation*}
y_s>y_{\infty}, \quad y'_s< 0, \quad  y''_s>0.
\end{equation*}
This implies that there exists $\bar y \coloneqq \displaystyle\lim_{s\to \infty} y_{s}\geq y_\infty$ and that $\displaystyle\lim_{s\to\infty} y'_{s}=0.$
However, it cannot be that $\bar y>y_{\infty}$, as it would imply, using~\eqref{ODEy}, that $\displaystyle \lim_{s \to \infty} y_s'' > 0$, contradicting the fact that $y_s' \to 0$, as $s \to +\infty$. Therefore, $\bar y = y_\infty$. \qedhere
%
%

%
%
%
%
%
%
\end{proof}

We are now ready to establish existence and uniqueness of an equilibrium for the mean-field game.

\begin{theorem}
\label{prop:MF_exist_uniq_infinite}
For each fixed random variable $\xi$ satisfying Assumption~\ref{hp:xi}, there exists a unique equilibrium $(\widehat \bfu, \widehat{q})$ for the mean-field game with infinite time horizon, among all equilibria {$(\widetilde \bfu, \widetilde q) \in \cU_\infty \times \cQ_\infty$.}

More precisely, $\widehat \bfu = z^{(\infty,\widehat{q})}$, where $z^{(\infty,q)}$ is the function defined in~\eqref{eq:z_infinite}, and $\widehat{q}$ is the unique solution to~\eqref{eq:IDE}, with initial condition $x = \bbE[\xi]$.
\end{theorem}

\begin{proof}
Let $\widehat{q}$ be the unique solution to~\eqref{eq:IDE}, with $x = \bbE[\xi]$. We need to check that {$\widehat q \in \cQ_\infty$.} We have the following three cases.

\medskip

\textit{Case $x = y_\infty$.} From Theorem~\ref{th:fixed_point_infinite}, we know that $\widehat{q}_s = y_\infty$, for all $s \geq 0$. Therefore,
\begin{equation}
\label{eq:z_bound}
z^{(\infty,\widehat{q})}_s = K \coloneqq \dfrac{y_\infty^{-\beta}}{\rho+\delta}, \quad \forall s \geq 0 \, ,
\end{equation}
and hence {$\widehat q \in \cQ_\infty$.}

\medskip

\textit{Case $x > y_\infty$.} From Theorem~\ref{th:fixed_point_infinite}, we know that $\widehat{q}$ is monotone decreasing and converges to $y_\infty$. This implies that, $0 < y_\infty < \widehat{q}_s \leq x$, for all $s \geq 0$, and that $z^{(\infty, \widehat{q})}$ is positive and bounded above by the constant $K$ defined in~\eqref{eq:z_bound}. Therefore, {$\widehat q \in \cQ_\infty$.}

\medskip

\textit{Case $x < y_\infty$.} From Theorem~\ref{th:fixed_point_infinite}, we know that $\widehat{q}$ is monotone increasing and converges to $y_\infty$. This implies that, $0 < x \leq \widehat{q}_s < y_\infty$, for all $s \geq 0$, and that $z^{(\infty, \widehat{q})}$ is positive and bounded above by $\frac{x^{-\beta}}{\rho+\delta}$. Therefore, {$\widehat q \in \cQ_\infty$.}

\medskip

Applying Theorem~\ref{th:MF_IDE}-\ref{th:MF_IDE_infinite_onlyif} we get the result. 
\end{proof}
{
A crucial question regarding the robustness of the mean-field equilibrium is whether the aggregate behavior of the agents remains consistent when passing from a large but finite horizon to an infinite one. Thanks to the characterization of the equilibrium provided by Theorem \ref{th:MF_IDE} and the analytical results obtained in Theorem \ref{th:fixed_point_infinite}, we can formally establish that the finite-horizon equilibrium $\widehat{q}^{(T)}$ converges to the unique infinite-horizon equilibrium $\widehat{q}^{(\infty)}$. This result, summarized in the following corollary, ensures the stability of the model's predictions with respect to the time horizon $T$ and justifies the use of the infinite-horizon case as a structural limit for long-term strategic interactions.

\begin{corollary}[Convergence of the Mean-Field Equilibrium]
\label{cor:convergence_equilibrium}
For each $T > 0$, let $(\widehat{\mathbf{u}}^{(T)}, \widehat{q}^{(T)})$ be the unique equilibrium of the mean-field game with finite horizon $T$, where $\widehat{q}^{(T)}$ is the unique solution to \eqref{eq:IDE}. Under the assumptions of Theorem \ref{th:fixed_point_infinite}, as $T \to +\infty$, the sequence of equilibrium trajectories $\{\widehat{q}^{(T)}\}_{T>0}$ converges to the unique infinite-horizon equilibrium $\widehat{q}^{(\infty)}$ in the weighted space $\dL^1_{\rho+2\delta}([0,\infty);\mathbb{R})$. Consequently, the optimal strategies $\widehat{\mathbf{u}}^{(T)}$ converge to the infinite-horizon optimal strategy $\widehat{\mathbf{u}}^{(\infty)}$ in the same sense.
\end{corollary}

\begin{proof}
The proof of Theorem \ref{th:fixed_point_infinite} shows that the family of transformed solutions $\cZ \coloneqq \{z^{(T)}\}_{T > 0}$ is relatively compact in $\dL^1_{\rho+2\delta}$ by means of the Fréchet-Kolmogorov theorem. 
Let $t_n \to \infty$ be an arbitrary sequence, consider the corresponding sequence $\{z^{(t_n)}\}$, and take any subsequence $\{z^{(t_{n_k})}\}$. By the existence part of the proof of Theorem \ref{th:fixed_point_infinite}, its limit must solve the infinite-horizon integro-differential equation. Since the uniqueness part of the same Theorem guarantees that such a solution is unique within the admissible set $\cC_{x,\infty}$, we conclude that $z^{(t_{n_k})} \to z^{(\infty)}$. 
Since this holds for an arbitrary sequence $t_n \to \infty$, we conclude that the entire family $\cZ$ converges to $z^{(\infty)}$ as $T \to \infty$. By the continuous mapping between the auxiliary function $z$ and the equilibrium trajectory $\widehat{q}$ (and the corresponding optimal control $\widehat{\mathbf{u}}$), the convergence of the equilibrium follows.
\end{proof}}

\begin{remark}\label{rem:extension2_t}
Also in this case the results of this section can be easily extended to the case where we consider an initial time $t > 0$. 

Clearly, we need to adapt (in an obvious way) the definition of equilibrium given in Definition~\ref{def:equilibrium} to include the initial time $t$.
The statement of Theorem~\ref{th:MF_IDE} is also adapted accordingly.

The integro-differential equation~\eqref{eq:IDE} becomes
\begin{align}\label{eq:IDE_t}
\begin{dcases}
\dfrac{\dd}{\dd s} y_s = -\delta y_s + \int_s^T \de^{-(\rho+\delta)(u-s)} y_u^{-\beta} \, \dd u, &s \in [t,T], \\
y_t = x \, ,
\end{dcases}
\end{align}
while equation~\eqref{eq:zIDE} becomes
\begin{equation}\label{eq:zIDE_t}
\begin{dcases}
\dfrac{\dd}{\dd s} z_s = \de^{(\rho+2\delta)s} \int_s^T \de^{-(\rho+\delta-\delta\beta)u} z_u^{-\beta} \, \dd u, &s \in [t,T], \\
z_t = \de^{\delta t}x\, .
\end{dcases}
\end{equation}
The sets in which we look for solutions to~\eqref{eq:zIDE_t} (now dependent on $t$) are
\begin{equation}
\label{eq:zIDE_solution_set_t}
\cC_{t,x} \coloneqq \{f \in \dC([t,T];(0,+\infty)) \colon \de^{\delta t}\underline \alpha^{x}_s \leq f_s \leq \de^{\delta t}\overline \alpha^{x}_s, \, \forall s \in [t,T]\},
\end{equation}
in the finite time horizon case, and
\begin{equation}
\label{eq:zIDE_solution_set_infinite_t}
\cC_{t,x,\infty} \coloneqq \{f \in \dC([t,+\infty);(0,+\infty)) \colon \de^{\delta t} \underline \alpha^{x, \infty}_s \leq f_s \leq \de^{\delta t}\overline \alpha^{x,\infty}_s, \, \forall s \geq t\},
\end{equation}
in the infinite time horizon case.

The statements of Theorems~\ref{th:fixed_point_finite}, \ref{prop:MF_exist_uniq_finite}, \ref{th:fixed_point_infinite}, and~\ref{prop:MF_exist_uniq_infinite}, remain the same (except for minor modifications). The bounds given in~\eqref{eq:IDE_solution_bounds} become
\begin{equation}
\label{eq:IDE_solution_bounds_t}
\de^{-\delta (s-t)} x \leq y_s \leq \de^{-\delta (s-t)} \overline \alpha^{x}_s, \qquad s \in [t,T] \, ,
\end{equation}
while those given in~\eqref{eq:IDE_solution_bounds_infinite} become
\begin{equation}
\label{eq:IDE_solution_bounds_infinite_t}
\de^{-\delta (s-t)} x \leq y_s \leq \de^{-\delta (s-t)} \left[x - \dfrac{x^{-\beta}}{\delta(1+\beta)(\rho+\delta-\delta\beta)} \right] + \dfrac{\de^{\delta\beta (s-t)} x^{-\beta}}{\delta(1+\beta)(\rho+\delta-\delta\beta)}, \quad s \geq t.
\end{equation}
\end{remark}

{{

\subsubsection{Economic Interpretation of the Stationary equilibrium}
\label{sec:economics-stationary}

The explicit nature of the stationary solution in the infinite horizon problem allows for a clear economic characterization of the long-term market equilibrium. In this setting, the unique stationary average production level $y_\infty$ is the positive solution to the algebraic equation $(\rho + \delta)\delta y - y^{-\beta} = 0$ (cf.\ \eqref{eq:segno}), given by:
\begin{equation*}
y_\infty = \left[ \frac{1}{\delta(\rho + \delta)} \right]^{\frac{1}{1+\beta}}.
\end{equation*}
Some comments on the economics of this result are in place.
\smallskip

\textbf{Sensitivity to parameters.} The term $\delta(\rho + \delta)$ can be interpreted as the \textit{effective user cost of capital}. It combines the physical depreciation rate $\delta$ with the opportunity cost of capital $\rho$. As either the discount rate $\rho$ or the depreciation rate $\delta$ increases, the stationary production capacity $y_\infty$ decreases. This reflects the intuition that firms are less willing to maintain high production levels when the future is heavily discounted or when maintaining capacity requires more intensive reinvestment. Furthermore, the parameter $\beta$ represents the inverse of the price elasticity of demand. A higher $\beta$ (less elastic demand) leads to a lower $y_\infty$, as firms strategically limit capacity to avoid the sharp price decreases associated with high aggregate supply in inelastic markets.
\smallskip

\textbf{Market stability and convergence.} The monotone convergence result of Theorem \ref{th:fixed_point_infinite} provides further economic insight. If the initial production level exceeds $y_\infty$, the equilibrium 
path decreases monotonically toward the steady state; if it starts below $y_\infty$, it increases monotonically. Thus, the stationary equilibrium is globally stable and dynamically attractive. Deviations from the steady state 
are gradually absorbed through investment adjustments without oscillations or overshooting, suggesting that competitive investment incentives dampen imbalances rather than amplify them. Notably, the stationary level $y_\infty$ and the trajectory of the average production $\widehat{q}$ are independent of the volatility parameter $\sigma$. This suggests that while individual firms face stochastic fluctuations, the "mean-field" or aggregate economy follows a deterministic and stable path toward equilibrium. We refer also to Section \ref{sec:deterministic} below on this.
}}

\section{The deterministic model}
\label{sec:deterministic}

{{In this section we discuss the deterministic version of the mean-field game problem introduced in Section~\ref{sec:model}. The structure of the problem will allow us to exploit the results given in Sections~\ref{sec:optpb} and~\ref{sec:fixed_point}, to get the unique equilibrium of the mean-field game in an explicit form. Indeed, the results obtained for the stochastic model reveal that the average production capacity and the optimal control at equilibrium do not depend on the parameter $\sigma$ appearing in SDE~\eqref{eq:SDE}, which describes the sensitivity of the production capacity with respect to random shocks. Thus, by formulating a deterministic counterpart of the problem studied in the previous sections, i.e., formally setting $\sigma = 0$ in the dynamics of the production capacity, we are able to show that results similar to those proved in Sections~\ref{sec:optpb} and~\ref{sec:fixed_point} are valid also in this context.

We emphasize that no absolute continuity assumption on the initial distribution is needed in the deterministic model. The evolution of the population distribution will be described directly through the pushforward of the initial law by the deterministic flow. This formulation naturally includes, in particular, Dirac initial distributions.

Let us consider the ordinary differential equation (ODE)
\begin{equation}
\label{eq:ODE}
\left\{
\begin{aligned}
&\dfrac{\dd}{\dd s} X_s = -\delta X_s + u_s, \quad s \in (0,T], \\
&X_0 = x,
\end{aligned}
\right.
\end{equation}
where $\delta > 0$ is a given coefficient and $\bfu \coloneqq (u_s)_{s \in [0,T]}$ is the control, chosen in a suitable class of admissible non-negative functions that will be specified below.

Note that, for any $x \in \R$ and any $\bfu \in \dL^1_{\mathrm{loc}}([0,T]; [0,+\infty))$, Equation~\eqref{eq:ODE} has a unique solution $X^{x,\bfu}$, given by
\begin{equation}
\label{eq:ODEsol}
X_s^{x,\bfu} = \de^{-\delta s} x + \int_0^s \de^{-\delta(s-r)} u_r \, \dd r, \quad s \in [0,T].
\end{equation}
As in the stochastic case, $X$ describes the evolution of the production capacity of a representative agent, which depreciates at a rate $\delta$ and can be increased by choosing the investment rate $\bfu$. Also in the deterministic case we are in an irreversible investment decisions setting, given the non-negativity of the control.

Next, we consider the discounted net profit functional
\begin{equation}
\label{eq:netprofitfunct_det}
J_{T,q}(x, \bfu) \coloneqq \int_0^T \de^{-\rho s} \left\{X_s^{x,\bfu} q_s^{-\beta} - \dfrac 12 u_s^2\right\} \dd s,
\end{equation}
where $\rho > 0$ is a discount factor, $\beta > 0$ is a fixed parameter, and $q = (q_s)_{s \in [0,T]}$ is a given deterministic measurable function.
The control $\bfu$ is chosen in either of the following two classes of admissible controls: if $T < +\infty$,
\begin{align}
\label{eq:admctlr_finite_det}
  \cU_{T} \coloneqq
  \biggl\{\bfu \colon [0,T] \to [0,+\infty) \text{ measurable and such that } \int_0^T u_s^2 \, \dd s < +\infty\biggr\};
\end{align}
if, instead, $T = +\infty$,
\begin{align}
\label{eq:admctlr_infinite_det}
  \cU_{\infty} \coloneqq
  \biggl\{\bfu \colon [0,\infty) \to [0,+\infty) \text{ measurable and such that } \int_0^s u_r \, \dd r < +\infty, \, \forall s \geq 0, \notag \\
  \qquad \text{and } \int_0^{+\infty} \de^{-\rho s} u_s^2 \, \dd s < +\infty\biggr\}.
\end{align}

From an economic point of view, the deterministic model captures an economy in which the production capacity of the representative firm is not affected by random shocks. Its evolution over time is, thus, uniquely determined by the investment efforts to increase it.

We introduce the following assumption.
\begin{assumption}
\label{hp:nu0}
The initial distribution $\nu_0$ of the agents in the economy belongs to $\cP_1((0,+\infty))$, namely,
\begin{equation*}
\nu_0((0,+\infty)) = 1,
\qquad
\int_0^{+\infty} x \, \nu_0(\dd x) < +\infty .
\end{equation*}
\end{assumption}

Here and in what follows, $\cP(E)$ denotes the set of probability measures on a Borel set $E\subseteq\R$, and $\cP_1(E)$ denotes the subset of probability measures with finite first moment. We identify $\nu_0 \in \cP_1((0,+\infty))$ with its extension to $\R$ concentrated on $(0,+\infty)$. If $\Phi \colon \R \to \R$ is Borel-measurable and $\nu \in \cP(\R)$, we denote by $\Phi_\# \nu$ the pushforward of $\nu$ through $\Phi$, i.e., the measure given, for any Borel set $E\subseteq\R$, by
\begin{equation*}
    \Phi_\# \nu(E) \coloneqq \int_\R \ind_E(\Phi(x)) \, \nu(\dd x) = \nu(\Phi^{-1}(E)).
\end{equation*}

Since the representative agent is chosen randomly by picking her/his initial production capacity level $x$ according to the initial distribution $\nu_0$, Assumption~\ref{hp:nu0} ensures that, for any $\bfu \in \cU_T$, $X_s^{x,\bfu} > 0$, for all $s \in [0,T]$ and for $\nu_0$-a.e. $x > 0$. Thus, the production capacity level of the representative agent is never negative.

To describe the evolution of the population distribution, we consider the flow associated with~\eqref{eq:ODE} on the whole real line. This avoids introducing any boundary condition at zero; since the initial distribution is concentrated on $(0,+\infty)$ and the control is non-negative, the resulting law remains concentrated on $(0,+\infty)$. For any $\bfu \in \cU_T$, set
\begin{equation}
\label{eq:flow_det}
\Phi_s^\bfu(x)
\coloneqq
\de^{-\delta s} x + \int_0^s \de^{-\delta(s-r)} u_r \, \dd r,
\qquad
(s,x) \in [0,T]\times\R.
\end{equation}
Then the population distribution at time $s$ is defined by
\begin{equation}
\label{eq:pushforward_det}
\nu_s^{\nu_0,\bfu}
\coloneqq
(\Phi_s^\bfu)_\# \nu_0,
\qquad s \in [0,T].
\end{equation}
Equivalently, for every bounded Borel-measurable function $\varphi \colon \R \to \R$,
\begin{equation}
\label{eq:pushforward_test}
\int_\R \varphi(x) \, \nu_s^{\nu_0,\bfu}(\dd x)
=
\int_\R \varphi(\Phi_s^\bfu(x)) \, \nu_0(\dd x).
\end{equation}

\begin{definition}
\label{def:weak_sol_cont_measure}
A narrowly continuous family $\nu=(\nu_s)_{s\in[0,T]}$ in $\cP_1(\R)$ is a \emph{weak measure-valued solution} to the continuity equation associated with~\eqref{eq:ODE}, with initial condition $\nu_0$, if, for every $\phi \in \dC_c^\infty([0,T]\times\R)$ and every $t\in[0,T]$,
\begin{align}
\label{eq:weak_sol_cont_measure_def}
&\int_\R \phi(t,x) \, \nu_t(\dd x) \notag \\
&\quad =
\int_\R \phi(0,x) \, \nu_0(\dd x)
+ \int_0^t \int_\R
\left\{
\dfrac{\partial}{\partial s} \phi(s,x)
+ [-\delta x + u_s] \dfrac{\partial}{\partial x} \phi(s,x)
\right\}
\nu_s(\dd x) \, \dd s .
\end{align}
\end{definition}

The family $\nu^{\nu_0,\bfu}=(\nu_s^{\nu_0,\bfu})_{s\in[0,T]}$ is the measure-valued solution to the continuity equation associated with~\eqref{eq:ODE}. The next lemma records this fact and the corresponding first-moment identity.

\begin{lemma}
\label{lem:cont_exist_uniq}
For any $\bfu \in \cU_T$ and any $\nu_0 \in \cP_1((0,+\infty))$, the family $\nu^{\nu_0,\bfu}$ defined by~\eqref{eq:pushforward_det} is the unique weak measure-valued solution to the continuity equation associated with~\eqref{eq:ODE}, in the sense of Definition~\ref{def:weak_sol_cont_measure}. More precisely, for every $\phi \in \dC_c^\infty([0,T]\times\R)$ and every $t\in[0,T]$,
\begin{align}
\label{eq:weak_sol_cont_measure}
&\int_\R \phi(t,x) \, \nu_t^{\nu_0,\bfu}(\dd x) \notag \\
&\quad =
\int_\R \phi(0,x) \, \nu_0(\dd x)
+ \int_0^t \int_\R
\left\{
\dfrac{\partial}{\partial s} \phi(s,x)
+ [-\delta x + u_s] \dfrac{\partial}{\partial x} \phi(s,x)
\right\}
\nu_s^{\nu_0,\bfu}(\dd x) \, \dd s .
\end{align}
Moreover, $\nu_s^{\nu_0,\bfu}((0,+\infty))=1$ for every $s\in[0,T]$, and $\nu_s^{\nu_0,\bfu}$ has finite first moment, given by
\begin{equation}
\label{eq:cont_avg}
\int_0^{+\infty} x \, \nu_s^{\nu_0,\bfu}(\dd x)
=
\de^{-\delta s} \int_0^{+\infty} x \, \nu_0(\dd x)
+ \int_0^s \de^{-\delta(s-r)} u_r \, \dd r,
\qquad s \in [0,T].
\end{equation}
\end{lemma}

\begin{proof}
Fix $\bfu \in \cU_T$ and $\nu_0 \in \cP_1((0,+\infty))$. For every $x\in\R$, the map $s\mapsto \Phi_s^\bfu(x)$ is absolutely continuous and solves~\eqref{eq:ODE}. Hence, for every $\phi \in \dC_c^\infty([0,T]\times\R)$ and every $x\in\R$, the chain rule gives, for Lebesgue-a.e. $s\in[0,T]$,
\begin{equation*}
\frac{\dd}{\dd s}\phi(s,\Phi_s^\bfu(x))
=
\partial_s\phi(s,\Phi_s^\bfu(x))
+[-\delta \Phi_s^\bfu(x)+u_s]\partial_x\phi(s,\Phi_s^\bfu(x)).
\end{equation*}
The narrow continuity of $s\mapsto \nu_s^{\nu_0,\bfu}$ follows from~\eqref{eq:pushforward_test}, the continuity of $s\mapsto \Phi_s^\bfu(x)$, and dominated convergence. Integrating the previous identity over $[0,t]$ and then with respect to $\nu_0(\dd x)$, and using the definition of pushforward, yields~\eqref{eq:weak_sol_cont_measure_def}, hence~\eqref{eq:weak_sol_cont_measure}. Since $\Phi_s^\bfu(x)>0$ whenever $x>0$, we also have $\nu_s^{\nu_0,\bfu}((0,+\infty))=1$.

Conversely, uniqueness follows from the uniqueness of the characteristics associated with the vector field $(s,x)\mapsto -\delta x+u_s$ on $\R$ and the standard characterization of measure-valued solutions to the continuity equation by pushforward along the flow. In the present one-dimensional affine case this characterization follows directly from the explicit formula~\eqref{eq:flow_det}.

Finally, applying~\eqref{eq:pushforward_test} with $\varphi(x)=x$ and using~\eqref{eq:flow_det}, we obtain
\begin{align*}
\int_\R x \, \nu_s^{\nu_0,\bfu}(\dd x)
&=
\int_\R \Phi_s^\bfu(x) \, \nu_0(\dd x) \\
&=
\de^{-\delta s} \int_\R x \, \nu_0(\dd x)
+ \int_0^s \de^{-\delta(s-r)} u_r \, \dd r,
\end{align*}
which is finite thanks to Assumption~\ref{hp:nu0} and the admissibility of $\bfu$.
\end{proof}

We are now ready to state the definition of equilibrium for our deterministic mean-field game.
\begin{definition}
\label{def:equilibrium_det}
Fix an initial distribution $\nu_0$ under Assumption~\ref{hp:nu0}.

A pair $(\widehat \bfu, \widehat{q})$, where $\widehat \bfu \in \cU_T$ and $\widehat{q} \colon [0,T] \to (0,+\infty)$ is a measurable function, is an \emph{equilibrium} of the deterministic mean-field game if
\begin{enumerate}[label=(\roman*)]
\item
$J_{T,\widehat{q}}(x, \widehat \bfu) \geq J_{T,\widehat{q}}(x, \bfu)$, for all $\bfu \in \cU_T$ and $\nu_0$-a.e. $x > 0$;
\item
\[
\widehat{q}_s
=
\int_0^{+\infty} x \, \nu_s^{\nu_0,\widehat \bfu}(\dd x),
\qquad s \in [0,T],
\]
where $\nu^{\nu_0,\widehat \bfu}$ is the pushforward flow of measures defined by~\eqref{eq:pushforward_det}.
\end{enumerate}
\end{definition}

Also in this case, we can adopt a fixed point argument to show that there exists a unique equilibrium, by solving, first, the problem of maximizing~\eqref{eq:netprofitfunct_det} for a given measurable function $q$, then determining the optimally controlled dynamics of the production capacity level of the representative agent and, finally, setting the fixed point argument from condition~(ii) in Definition~\ref{def:equilibrium_det}.

\subsection{The optimization problem}

Let us consider the optimization problem
\begin{equation}
\label{eq:V_det}
V_{T,q}(x) \coloneqq \sup_{\bfu \in \cU_{T}} J_{T,q}(x, \bfu), \quad x > 0,
\end{equation}
where $J_{T,q}$ is the discounted net profit functional defined in~\eqref{eq:netprofitfunct_det}, for any given and fixed measurable deterministic function $q$.

In the finite time horizon case, i.e. $T < +\infty$, the following result holds, which is analogous to Proposition~\ref{th:V_finite}.

\begin{proposition}
\label{th:V_finite_det}
Fix $x > 0$ and $q \in \cQ_T$, where $\cQ_T$ is the set defined in~\eqref{eq:QT}. Then, $\widehat \bfu^{(T,q)} \coloneqq z^{(T,q)} \in \cU_{T}$, where $z^{(T,q)}$ is the function defined in~\eqref{eq:z_finite}, is an optimal control for problem~\eqref{eq:V_det}.

Moreover, $\widehat \bfu^{(T,q)}$ is essentially unique, i.e., if $\overline \bfu^{(T,q)} \in \cU_{T}$ is an optimal control for problem~\eqref{eq:V_det} different from $\widehat \bfu^{(T,q)}$, then
\begin{equation*}
\overline u^{(T,q)}_s = \widehat u^{(T,q)}_s, \quad \text{for $\mathrm{Leb}$-a.e. } s \in [0,T].
\end{equation*}

Finally, the value function of the optimization problem admits the explicit expression
\begin{equation}
\label{eq:V_finite_explicit_det}
V_{T,q}(x) = x z^{(T,q)}_0 + \dfrac{1}{2} \int_0^T \de^{-\rho s} (z^{(T,q)}_s)^2 \, \dd s,
\end{equation}
and the optimally controlled state $X^{x, \widehat \bfu^{(T,q)}}$ is given by
\begin{equation}
\label{eq:X_opt_ctrl_finite_det}
X^{x, \widehat \bfu^{(T,q)}}_s
=
\de^{-\delta s} x
+ \int_0^s \de^{-\delta(s-r)} z^{(T,q)}_r \, \dd r,
\quad s \in [0,T].
\end{equation}
\end{proposition}

\begin{proof}
Replicating the argument of the proof of Lemma~\ref{prop:J_finite_properties}, we get that
\begin{equation*}
J_{T,q}(x,\bfu)
=
x z^{(T,q)}_0
+ \int_0^T \de^{-\rho s}
\left\{ z^{(T,q)}_s u_s - \dfrac 12 u_s^2\right\}
\dd s .
\end{equation*}
Then, the result follows using the same reasoning as in the proof of Proposition~\ref{th:V_finite}.
\end{proof}

In the infinite time horizon case, i.e. $T= +\infty$, we have the following statement, which is the deterministic counterpart of Proposition~\ref{th:V_infinite}. We state it without proof.

\begin{proposition}
\label{th:V_infinite_det}
Fix $x > 0$ and $q \in \cQ_\infty$, where $\cQ_\infty$ is the set defined in~\eqref{eq:Qinfty}. Then, $\widehat \bfu^{(\infty,q)} \coloneqq z^{(\infty,q)} \in \cU_{\infty}$, where $z^{(\infty,q)}$ is the function defined in~\eqref{eq:z_infinite}, is an optimal control for problem~\eqref{eq:V_det}.
Moreover, $\widehat \bfu^{(\infty,q)}$ is essentially unique, i.e., if $\overline \bfu^{(\infty,q)} \in \cU_{\infty}$ is an optimal control for problem~\eqref{eq:V_det} different from $\widehat \bfu^{(\infty,q)}$, then
\begin{equation*}
\overline u^{(\infty,q)}_s = \widehat u^{(\infty,q)}_s, \quad \text{for $\mathrm{Leb}$-a.e. } s \geq 0.
\end{equation*}

Finally, the value function of the optimization problem admits the explicit expression
\begin{equation}
\label{eq:V_infinite_explicit_det}
V_{\infty,q}(x) = x z^{(\infty,q)}_0 + \dfrac{1}{2} \int_0^{+\infty} \de^{-\rho s} (z^{(\infty,q)}_s)^2 \, \dd s,
\end{equation}
and the optimally controlled state $X^{x, \widehat \bfu^{(\infty,q)}}$ is given by
\begin{equation}
\label{eq:X_opt_ctrl_infinite_det}
X^{x, \widehat \bfu^{(\infty,q)}}_s
=
\de^{-\delta s} x
+ \int_0^s \de^{-\delta(s-r)} z^{(\infty,q)}_r \, \dd r,
\quad s \geq 0.
\end{equation}
\end{proposition}

\subsection{Existence and uniqueness of equilibria}

Also in the deterministic mean-field game previously introduced, the search for an equilibrium boils down to finding a fixed point of a suitable map. According to Definition~\ref{def:equilibrium_det}, this map is
\[
(q_s)_{s \in [0,T]}
\longmapsto
\left(
\int_0^{+\infty} x \, \nu_s^{\nu_0,\widehat \bfu^{(T,q)}}(\dd x)
\right)_{s \in [0,T]},
\]
where $\widehat \bfu^{(T,q)}$ is the optimal control for the optimization problem~\eqref{eq:V_det} -- whose expression is given either in Proposition~\ref{th:V_finite_det}, in the case $T < +\infty$, or in Proposition~\ref{th:V_infinite_det}, in the case $T = +\infty$ -- and $\nu^{\nu_0,\widehat \bfu^{(T,q)}}$ is the pushforward flow of measures defined by~\eqref{eq:pushforward_det}.

The next result, which is analogous to Theorem~\ref{th:MF_IDE}, shows that also in the deterministic case the fixed point map is the solution map of the same integral equation studied in Section~\ref{sec:fixed_point}. We omit its proof, which is a straightforward adaptation of the proof of Theorem~\ref{th:MF_IDE}; the only change is that the first moment of the population law is now computed through~\eqref{eq:cont_avg}.

\begin{theorem}
\label{th:MF_IDE_det}
Let us fix an initial distribution $\nu_0$ satisfying Assumption~\ref{hp:nu0}, and set
\[
\bar x_0 \coloneqq \int_0^{+\infty} x \, \nu_0(\dd x).
\]

Consider the deterministic mean-field game in the finite time horizon case, i.e., $T < +\infty$. Then,
\begin{enumerate}[label=(\roman*)]
\item If there exists an equilibrium $(\widehat \bfu, \widehat{q})$ in the sense of Definition~\ref{def:equilibrium_det}, such that $\widehat q \in \cQ_T$, then $\widehat{q}$ is a solution to the integral equation
\begin{equation}
\label{eq:IDEint_finite_det}
y_s
=
\de^{-\delta s} \bar x_0
+\int_0^s \de^{-\delta (s-r)}
\int_r^T \de^{-(\rho+\delta)(u-r)} y_u^{-\beta} \, \dd u \, \dd r,
\quad s \in [0,T].
\end{equation}
\item Vice versa, if there exists a unique solution $\widehat{q} \in \cQ_T$ to~\eqref{eq:IDEint_finite_det}, then there exists a unique equilibrium $(\widehat \bfu, \widehat{q}) = (z^{(T,\widehat q)},\widehat{q})$ of the mean-field game among all equilibria $(\widetilde \bfu, \widetilde q) \in \cU_T \times \cQ_T$, where $z^{(T,q)}$ is the function defined in~\eqref{eq:z_finite}.
\end{enumerate}
Consider, instead, the deterministic mean-field game in the infinite time horizon case, i.e., $T = +\infty$. Then,
\begin{enumerate}[label=(\roman*)]
\setcounter{enumi}{2}
\item If there exists an equilibrium $(\widehat \bfu, \widehat{q})$ in the sense of Definition~\ref{def:equilibrium_det}, such that $\widehat q \in \cQ_\infty$, then $\widehat{q}$ is a solution to the integral equation
\begin{equation}
\label{eq:IDEint_infinite_det}
y_s
=
\de^{-\delta s} \bar x_0
+\int_0^s \de^{-\delta (s-r)}
\int_r^{+\infty} \de^{-(\rho+\delta)(u-r)} y_u^{-\beta} \, \dd u \, \dd r,
\quad s \geq 0.
\end{equation}
\item Vice versa, if there exists a unique solution $\widehat{q} \in \cQ_\infty$ to~\eqref{eq:IDEint_infinite_det}, then there exists a unique equilibrium $(\widehat \bfu, \widehat{q}) = (z^{(\infty,\widehat{q})},\widehat{q})$ of the mean-field game among all equilibria $(\widetilde \bfu, \widetilde q) \in \cU_\infty \times \cQ_\infty$, where $z^{(\infty,q)}$ is the function defined in~\eqref{eq:z_infinite}.
\end{enumerate}
\end{theorem}

Finally, we get the following result, which establishes existence and uniqueness of an equilibrium also for the deterministic mean-field game. We omit its proof, since it follows the same lines of the proofs of Theorems~\ref{prop:MF_exist_uniq_finite} and~\ref{prop:MF_exist_uniq_infinite}.

\begin{theorem}
\label{prop:MF_exist_uniq_det}
For each fixed $\nu_0$ satisfying Assumption~\ref{hp:nu0}, there exists a unique equilibrium $(\widehat \bfu, \widehat{q})$ for the deterministic mean-field game with finite time horizon, respectively infinite time horizon, among all equilibria $(\widetilde \bfu, \widetilde q) \in \cU_T \times \cQ_T$, respectively $(\widetilde \bfu, \widetilde q) \in \cU_\infty \times \cQ_\infty$.

More precisely, in the finite time horizon case, $\widehat \bfu = z^{(T,\widehat{q})}$, where $z^{(T,\widehat{q})}$ is the function defined in~\eqref{eq:z_finite}; in the infinite time horizon case, $\widehat \bfu = z^{(\infty,\widehat{q})}$, where $z^{(\infty,\widehat{q})}$ is the function defined in~\eqref{eq:z_infinite}. In both cases, $\widehat{q}$ is the unique solution to~\eqref{eq:IDE}, with initial condition
\[
x = \int_0^{+\infty} y \, \nu_0(\dd y).
\]
In particular, the result applies to Dirac initial distributions $\nu_0=\delta_x$, for every $x>0$.
\end{theorem}

\begin{remark}
We observe that the equilibrium structure derived in the stochastic setting is invariant with respect to the diffusion parameter $\sigma$. This ensures a consistent transition to the deterministic case discussed in the present section. More precisely, for a fixed initial condition and along the same equilibrium control, the explicit representation of the optimally controlled state yields $\widehat{X}^{\sigma} \to \widehat{X}^{0}$ locally uniformly in time, in probability, as $\sigma \downarrow 0$; in fact, the convergence holds pathwise on compact time intervals along Brownian sample paths. The equilibrium mean production capacity $\widehat q$ and the equilibrium investment rate $\widehat \bfu$ are therefore the same objects obtained in the deterministic formulation.
\end{remark}
}}

{
\section{Numerical experiments}
\label{sec:num}

In this section we discuss some numerical simulations based on our results, to visualize the behaviour of the equilibrium average production capacity and of the optimal investment strategy at equilibrium, and at the same time highlighting the differences between the finite and infinite time horizon cases.

We consider the following values for the model parameters: $\beta = 2$, $\delta= 0.01$, $\rho = 0.03$. As stressed in the previous sections, the equilibrium average production capacity and the optimal investment strategy do not depend on volatility $\sigma$ appearing in SDE~\eqref{eq:SDE}. Nonetheless, we are going to discuss the effect of such parameter on the equilibrium production capacity.

\subsection{The finite time horizon case}
\label{sec:num:finite}
We consider three different time horizons, namely, $T = 3, 30, 300$. Thus, we compare the dynamics at equilibrium of the average production capacity $\widehat q$, given in~\eqref{eq:IDEint_finite}, and of the optimal investment strategy $\widehat \bfu$, which coincides with $z^{T,\widehat q}$ in~\eqref{eq:z_finite}, on a short, medium, and long time horizon.

We anticipate here that the longer horizon $T = 300$ is the same as the one used in the infinite-time horizon case, where it is used to truncate the time domain; this will allow us to compare the behaviour of $\widehat q$ and of $\widehat \bfu$ between the two cases. 
The specific choice of $T = 300$ will be motivated in Section~\ref{sec:num:infinite} below. 

First, for each of the three time horizons we solve numerically the integro-differential equation~\eqref{eq:zIDE}, whose solution $z$ is used to compute the equilibrium average production capacity $\widehat q$, via the transformation $\widehat q_s = \de^{-\delta s} z_s$, $s \in [0,T]$ (cf. {Lemma}~\ref{prop:IDE_transformation}); finally, we integrate numerically the expression appearing in~\eqref{eq:z_finite} to compute the equilibrium optimal investment strategy $\widehat \bfu$.

In Figure~\ref{fig:orizz_finito}, we plot the three functions computed above for each of the three time horizons, choosing as initial average production capacity $x = 10$; the graphs on each column show a comparison between the short, medium, and long time horizons.

\begin{figure}
    \centering
    \includegraphics[scale=0.37]{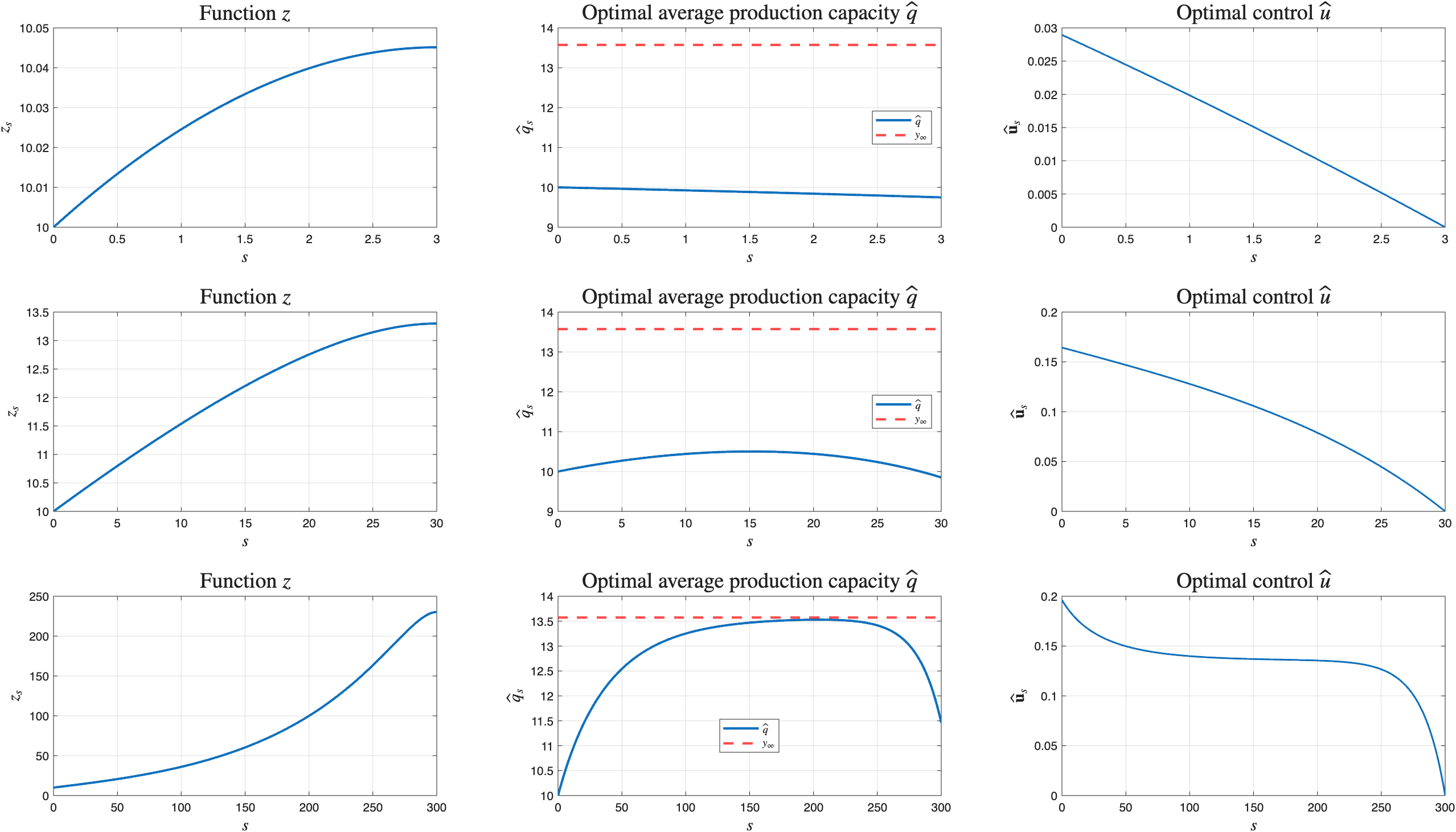}
    \caption{The solution $z$ to the integro-differential equation~\eqref{eq:zIDE}, the equilibrium average production capacity $\widehat q$, and the equilibrium optimal investment strategy $\widehat \bfu$, in the short, medium, and long time horizon cases. Initial condition $x = 10$.}
    \label{fig:orizz_finito}
\end{figure}

We observe that in the short time horizon case, the capital deterioration effect dominates the dynamics of $\widehat q$ and the investment level is relatively small. In the medium time horizon case, the investment effect predominates in the first half of the time interval, but then, as the investment effort decreases, the capital deterioration effect takes over. In the long time horizon case, the equilibrium average production capacity is pushed towards the infinite-time horizon steady state $y_\infty$ (cf. Theorem~\ref{th:fixed_point_infinite}) by the investment effort, which is almost kept constant for a large portion of the time interval.
However, we observe that, since the optimal control $\widehat \bfu$ is decreasing and verifies $\widehat u_T = 0$, the capital deterioration effect returns to dominate in the last part of the time interval also in the long horizon case.

We compare these graphs with those shown in Figure~\ref{fig:orizz_finito_equil}, where we compute the same three functions, starting at an initial production capacity $x = y_\infty$, i.e., at the steady state of the infinite time horizon case.

\begin{figure}
    \centering
    \includegraphics[scale=0.37]{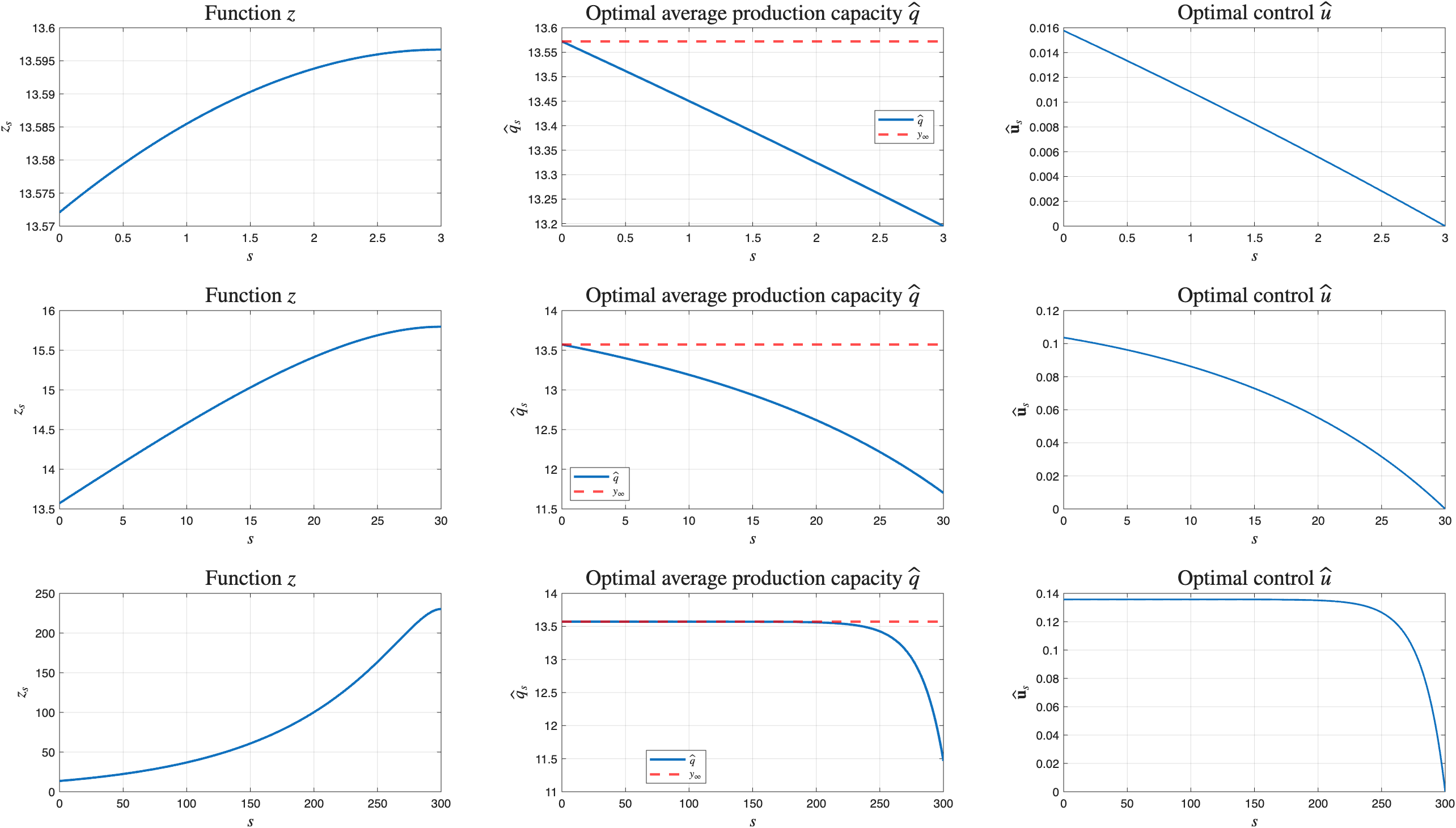}
    \caption{The solution $z$ to the integro-differential equation~\eqref{eq:zIDE}, the equilibrium average production capacity $\widehat q$, and the equilibrium optimal investment strategy $\widehat \bfu$, in the short, medium, and long time horizon cases. Initial condition $x = y_\infty \approx 13.5721$.}
    \label{fig:orizz_finito_equil}
\end{figure}

In this case, as the time horizon gets larger the capital deterioration effect is delayed by a higher investment level in the first part of the time interval. Nonetheless, the condition $\widehat u_T = 0$ implies that in the end the capital deterioration effect prevails.
{We observe a similar effect also in the case in which the initial production capacity $x$ is higher than $y_\infty$, as shown in Figure~\ref{fig:orizz_finito_sopra_yinfty}, where we set $x = 15$. Note that after a sufficiently large time, the average production capacity falls below the level $y_\infty$. We emphasize that in this case that the behaviour of the average production capacity and of the optimal control at equilibrium is different from the previous two cases, in which we observe that $\widehat q$ is concave and $\widehat \bfu$ is decreasing. In general, the properties of these two functions depend on the parameters of the model, but it is not clear how they determine the shape of $\widehat q$ and $\widehat \bfu$. This interesting question is, thus, left for future research.}

\begin{figure}
    \centering
    \includegraphics[scale=0.4]{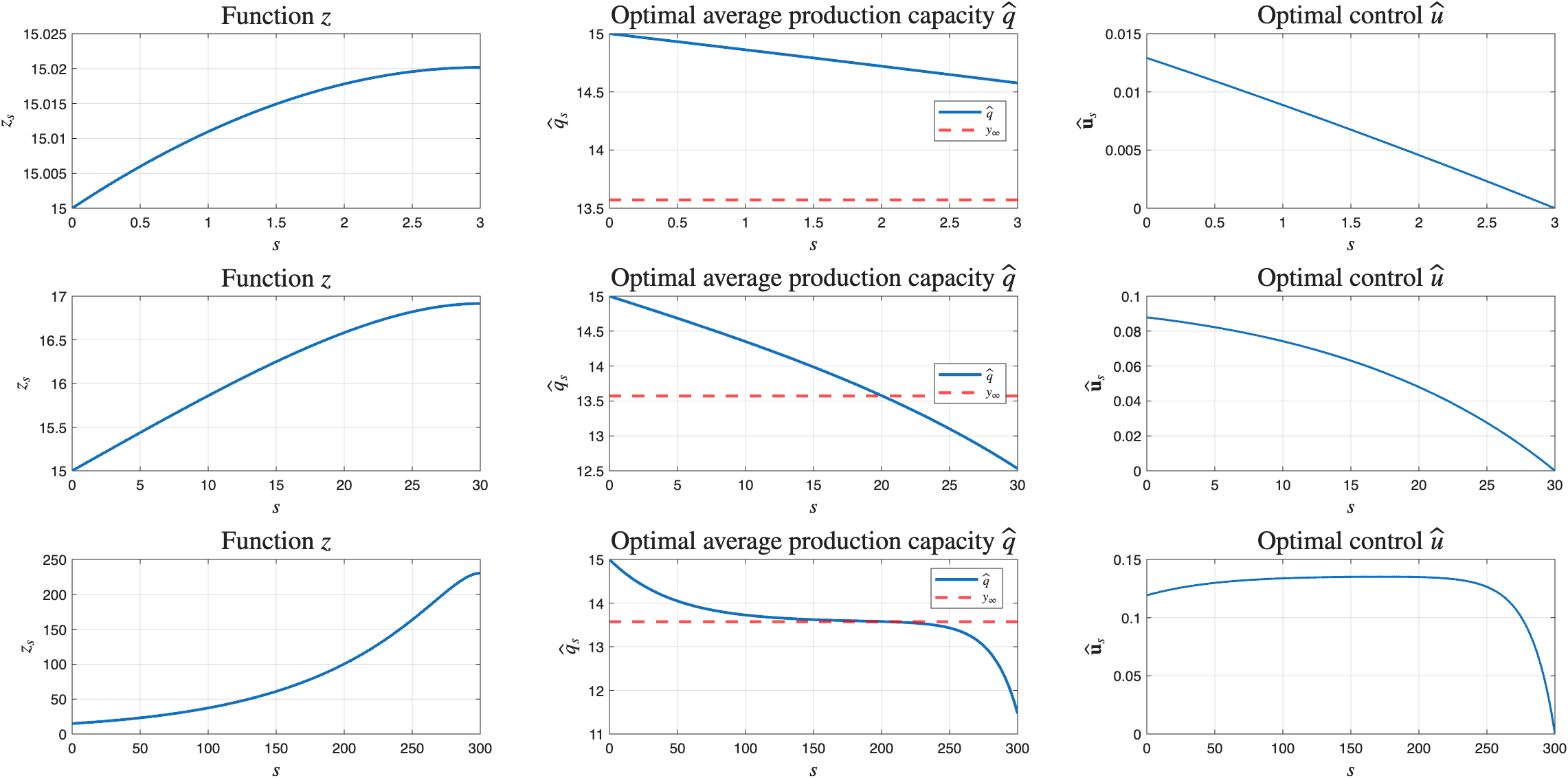}
    \caption{The solution $z$ to the integro-differential equation~\eqref{eq:zIDE}, the equilibrium average production capacity $\widehat q$, and the equilibrium optimal investment strategy $\widehat \bfu$, in the short, medium, and long time horizon cases. Initial condition $x = 15$.}
    \label{fig:orizz_finito_sopra_yinfty}
\end{figure}

{Figures~\ref{fig:orizz_finito}, \ref{fig:orizz_finito_equil}, and~\ref{fig:orizz_finito_sopra_yinfty}} clearly suggest that the finite time horizon $\widehat q$ and $\widehat \bfu$ converge to the corresponding functions in the infinite time horizon case.
However, a different numerical method needs to be used in the latter setting. Indeed, a simple truncation at time $T$ of the integro-differential equation~\eqref{eq:zIDE}, which is used to compute, in turn, $\widehat q$ and $\widehat \bfu$, would imply that $z'_T = 0$ (as it is possible to see in both Figures~\ref{fig:orizz_finito} and~\ref{fig:orizz_finito_equil}). Hence, even for a large value of $T$, we would see a decreasing pattern of $\widehat q$ in the long run, which is not consistent with the convergence to $y_\infty$ proved in Theorem~\ref{th:fixed_point_infinite}. We will discuss the details of such method in Section~\ref{sec:num:infinite} below.

We conclude our discussion of the finite time horizon case, by looking at {Figures~\ref{fig:traiett_orizz_finito}, \ref{fig:traiett_orizz_finito_equil}, and~\ref{fig:traiett_orizz_finito_sopra_yinfty}}, which provide sample trajectories of the equilibrium production capacity process $\widehat X \coloneqq X^{\xi, \widehat \bfu}$ (cf. Equation~\eqref{eq:SDE}), in each of the three time horizons given above and for $\sigma = 0.001, 0.01, 0.1$. In both figures we consider $\xi = \widehat X_ 0$ deterministic, thus coinciding with the initial average production capacity $x$; more precisely, in Figure~\ref{fig:traiett_orizz_finito} we consider $\xi = x = 10$; in Figure~\ref{fig:traiett_orizz_finito_equil}, we set $\xi = x = y_\infty$; {in Figure~\ref{fig:traiett_orizz_finito_sopra_yinfty}, we fix $\xi = x = 15$.}

\begin{figure}
    \centering
    \includegraphics[scale=0.37]{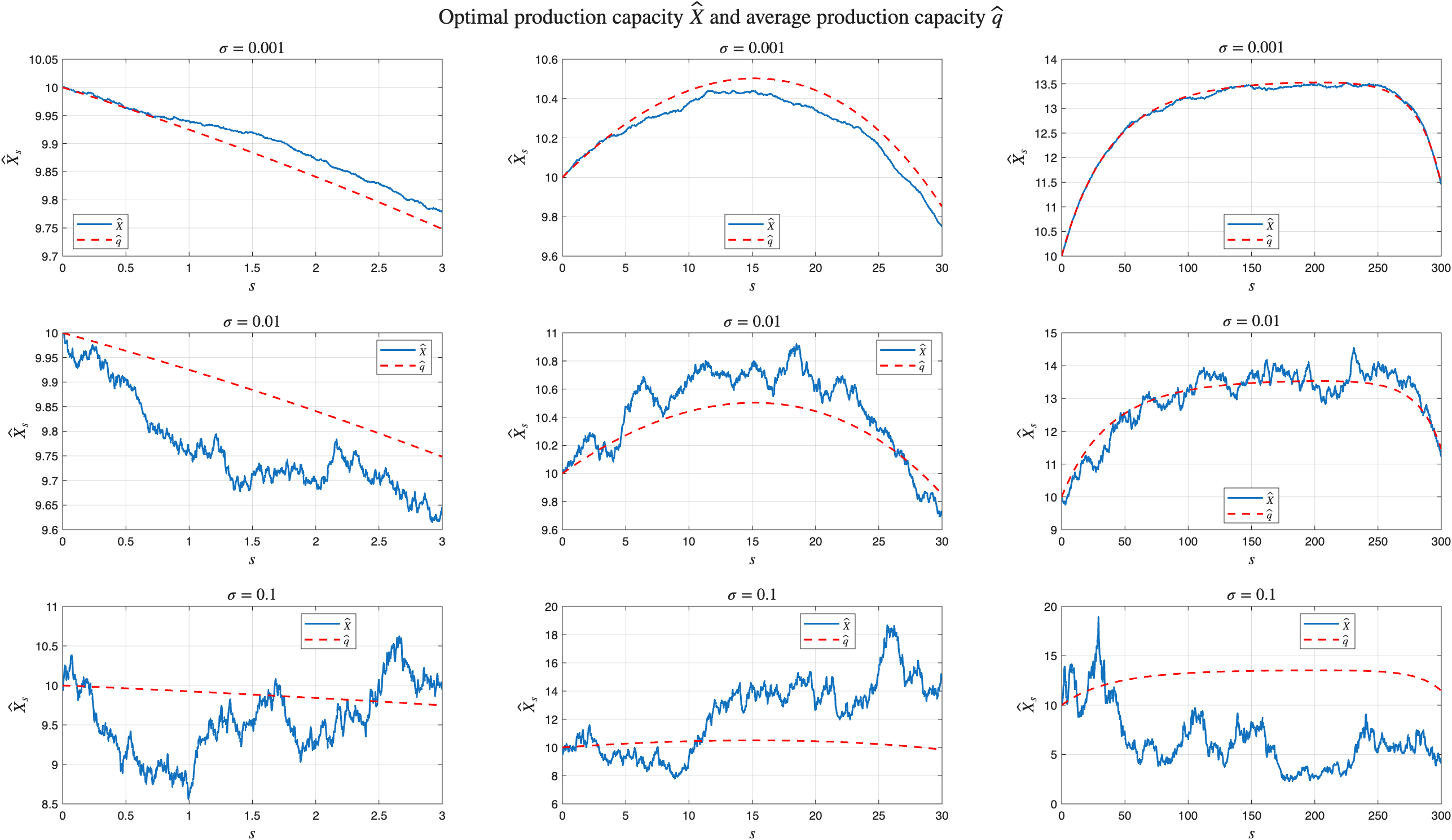}
    \caption{Trajectories of $\widehat X$ and the equilibrium average production capacity $\widehat q$, in the short, medium, and long time horizon cases, and for $\sigma = 0.001, 0.01, 0.1$. Initial condition $x = 10$.}
    \label{fig:traiett_orizz_finito}
\end{figure}

\begin{figure}
    \centering
    \includegraphics[scale=0.37]{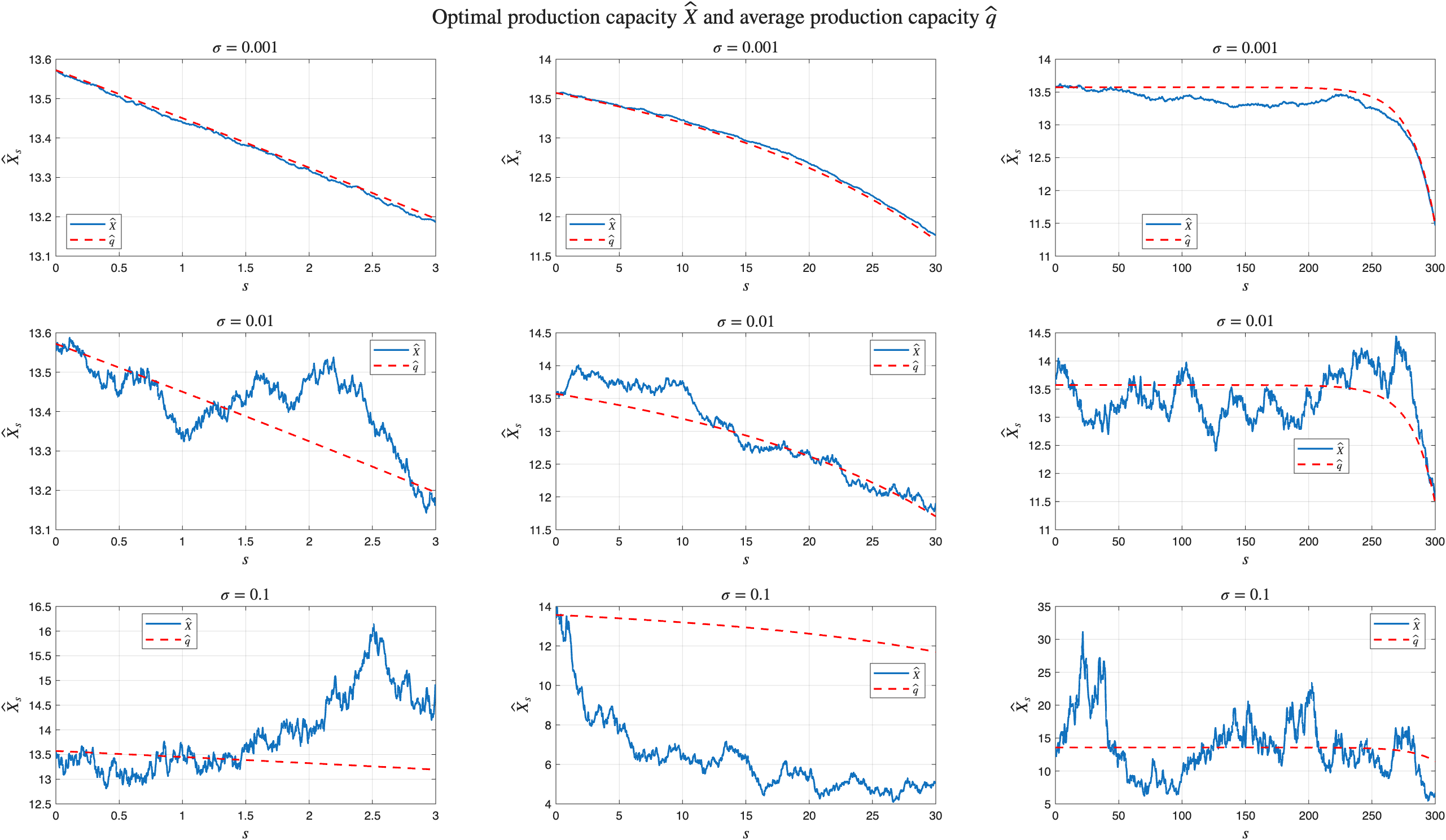}
    \caption{Trajectories of $\widehat X$ and the equilibrium average production capacity $\widehat q$, in the short, medium, and long time horizon cases, and for $\sigma = 0.001, 0.01, 0.1$. Initial condition $x = y_\infty \approx 13.5721$.}
    \label{fig:traiett_orizz_finito_equil}
\end{figure}

\begin{figure}
    \centering
    \includegraphics[scale=0.41]{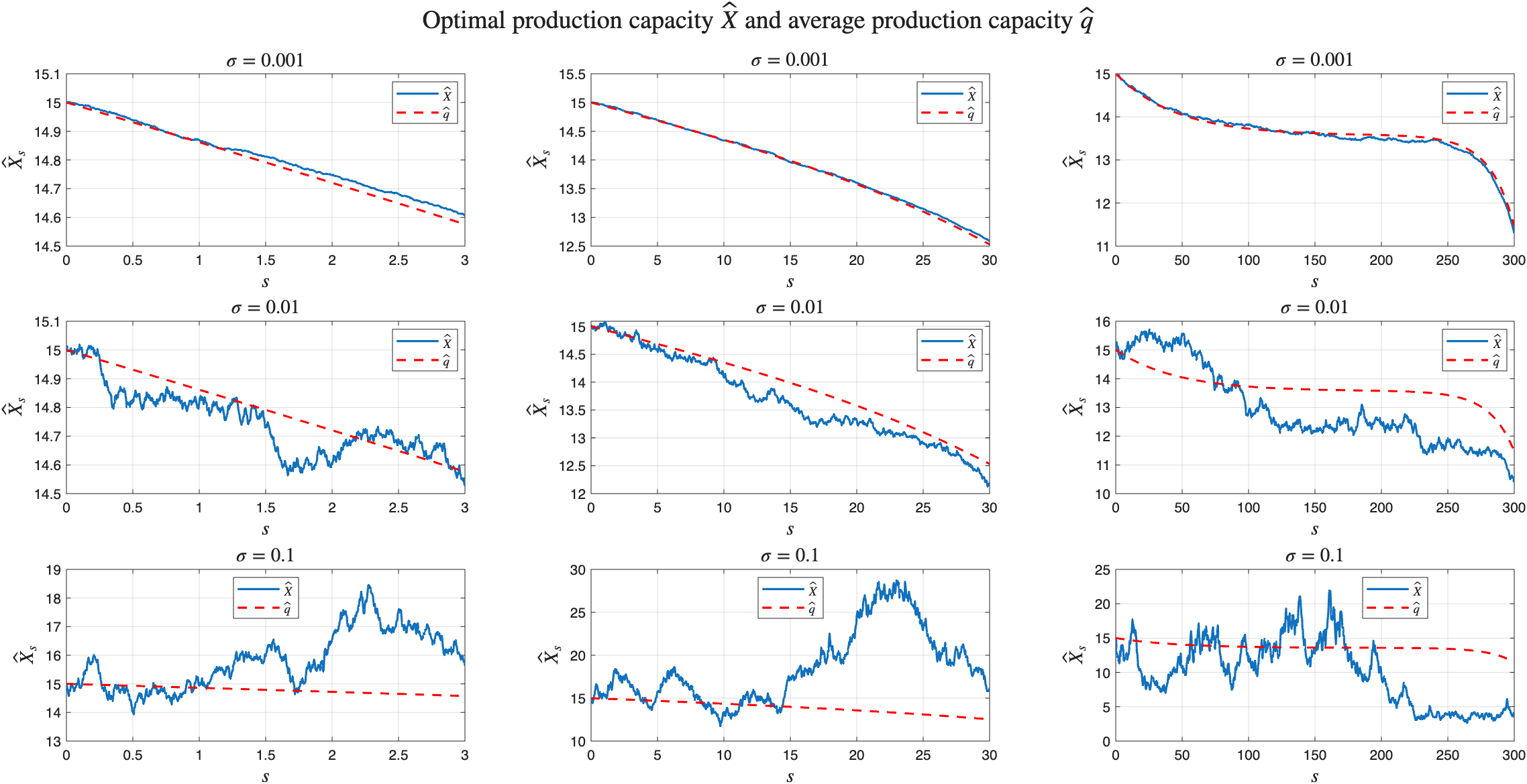}
    \caption{Trajectories of $\widehat X$ and the equilibrium average production capacity $\widehat q$, in the short, medium, and long time horizon cases, and for $\sigma = 0.001, 0.01, 0.1$. Initial condition $x = 15$.}
    \label{fig:traiett_orizz_finito_sopra_yinfty}
\end{figure}

As it is clear from our results, the role of $\sigma$ is just to measure the exposure of the production capacity level to the source of risk modeled by the Wiener process $B$ appearing in SDE~\eqref{eq:SDE}. Lower values of $\sigma$ imply that the equilibrium production capacity $\widehat X$ follows more closely its average, while higher values of $\sigma$ entail larger fluctuations, due to a higher sensitivity to risk.

\subsection{The infinite time horizon case}
\label{sec:num:infinite}
To compute numerically the optimal average production capacity $\widehat q$ and the optimal control $\widehat \bfu$ at equilibrium, we truncate the time domain $[0,+\infty)$ to the interval $[0,T]$, with $T = 300$.
The specific choice of $T = 300$ is motivated by a balance between accuracy and stability of the \emph{shooting method}, which is discussed below, and the need for a sufficiently long time horizon to show convergence of $\widehat q$ to the steady state $y_\infty$ given in Theorem~\ref{th:fixed_point_infinite}.
In addition, this value allows us to make a comparison with the longer finite time horizon case previously illustrated.

As discussed in Section~\ref{sec:num:finite}, we do not rely on a simple truncation to $T$ of the integro-differential equation~\eqref{eq:zIDE}, as this would artificially impose a decreasing behaviour of $\widehat q$, as time approaches $T$. This would be inconsistent with the convergence of the average production capacity to $y_\infty$, as stated in Theorem~\ref{th:fixed_point_infinite}.

We use, instead, the fact that any solution $z$ to the integro-differential equation~\eqref{eq:zIDE} also solves the second-order ODE~\eqref{eq:zIDE_infinitebis}, to which we can associate the initial value problem (IVP)
\begin{equation*}
\begin{dcases}
z''_s = (\rho+2\delta)z'_{s}- \de^{\delta (1+\beta)s}z_{s}^{-\beta}, &s \geq 0, \\
z_0 = x \, , \\
z'_0 =\zeta > 0 \, ,
\end{dcases}
\end{equation*}
where $x$ is the given initial average production capacity.
From the proof of Theorem~\ref{th:fixed_point_infinite}, we know that there exists a unique value of $\zeta > 0$ such that the IVP above has a unique solution in the class $\cC_{x,\infty}$ defined in~\eqref{eq:zIDE_solution_set_infinite}, which is also the unique solution to~\eqref{eq:zIDE}.

To approximate $z$ numerically we employ the so-called \emph{shooting method}, which, starting from an initial guess and using a sequence of successive approximations (\textit{via} the bisection method), computes the value of $\zeta$. This allows us to retain the correct properties of the solution $z$ to IDE~\eqref{eq:IDE} and, in particular, the fact that it grows exponentially at infinity at rate $\delta$ (cf. Figure~\ref{fig:orizz_infinito} below), a fact which is implied by the convergence of $\widehat q$ to $y_\infty$.

To compute the optimal control $\widehat \bfu$, we observe that truncating the integral in~\eqref{eq:z_infinite} at $T$, would force $\widehat u_T = 0$, which would yield a poor approximation as time approaches the truncation point. For this reason, we actually compute numerically the solution to IDE~\eqref{eq:IDE} -- and, thus, the average production capacity $\widehat q$ -- on the larger time interval $[0, \widetilde T]$, with $\widetilde T = 400$, and then we use the intended time interval $[0,T] = [0, 300]$ to compute $\widehat \bfu$ and to draw the graphs.

In Figure~\ref{fig:orizz_infinito} we compare the three functions $z$, $\widehat q$, and $\widehat \bfu$, choosing as initial conditions either $x = 10$, or $x = y_\infty + 0.01$, i.e., slightly above the steady state $y_\infty$.

\begin{figure}
    \centering
    \includegraphics[scale=0.37]{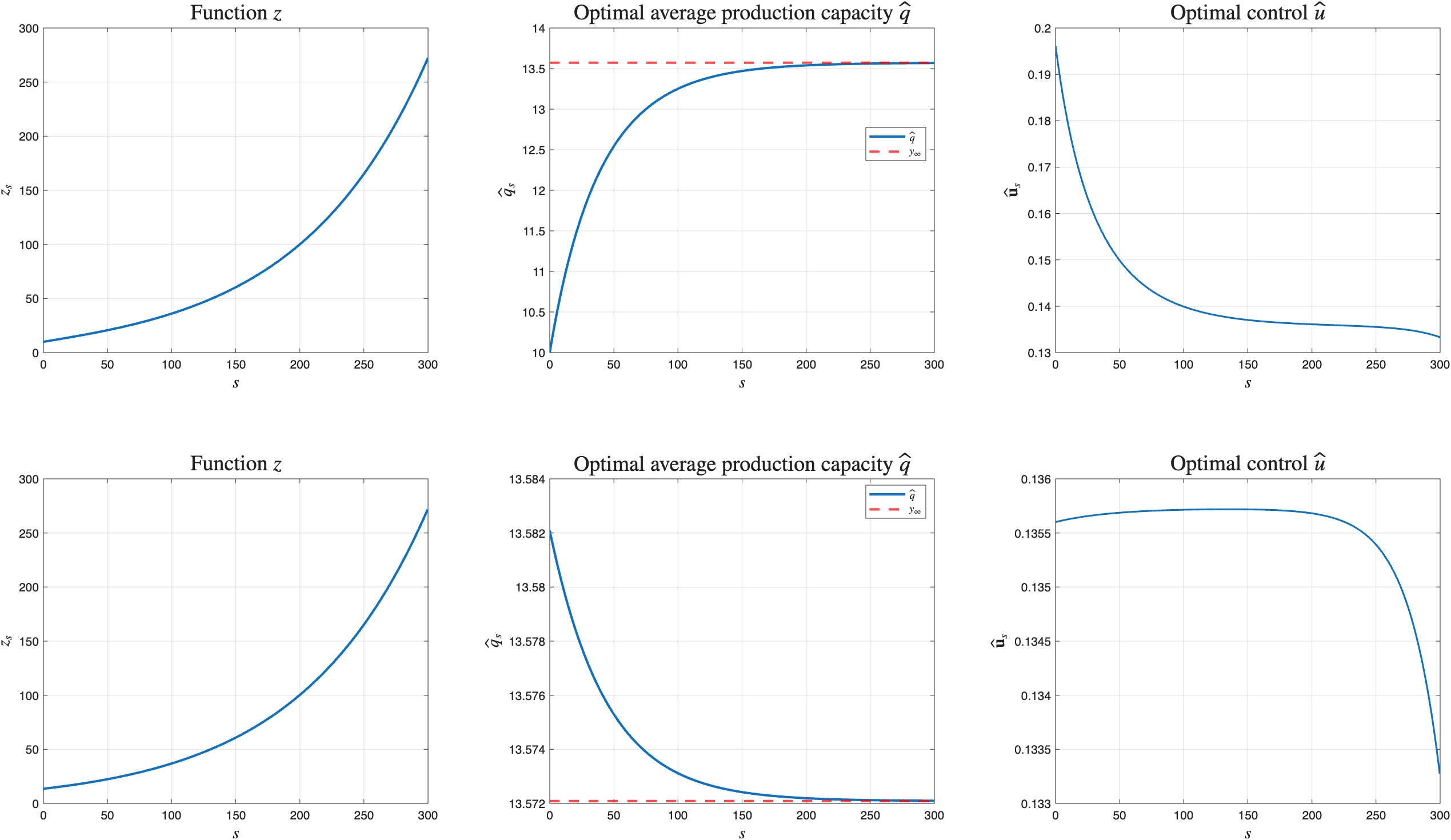}
    \caption{The solution $z$ to the integro-differential equation~\eqref{eq:zIDE}, the equilibrium average production capacity $\widehat q$, and the equilibrium optimal investment strategy $\widehat \bfu$. Initial conditions: $x = 10$ in the first row; $x = y_\infty + 0.01 \approx 13.5821$ in the second row.}
    \label{fig:orizz_infinito}
\end{figure}

As expected, we obtain convergence of the equilibrium average production capacity to the steady state $y_\infty$. The optimal control reaches an almost constant value in a large part of the time interval, before starting to decrease to zero, as implied by Equation~\eqref{eq:z_infinite}.

Also in the infinite time horizon case, we conclude our discussion with a simulation of sample trajectories of the equilibrium production capacity process $\widehat X \coloneqq X^{\xi, \widehat \bfu}$, whose dynamics are given in~\eqref{eq:SDE}, for $\sigma = 0.001, 0.01, 0.1$. In Figure~\ref{fig:traiett_orizz_infinito} we show two cases, in which we consider $\xi = \widehat X_ 0$ deterministic, thus coinciding with the initial average production capacity $x$; more precisely, $\xi = x = 10$ and $\xi = x = y_\infty + 0.01$.

\begin{figure}
    \centering
    \includegraphics[scale=0.37]{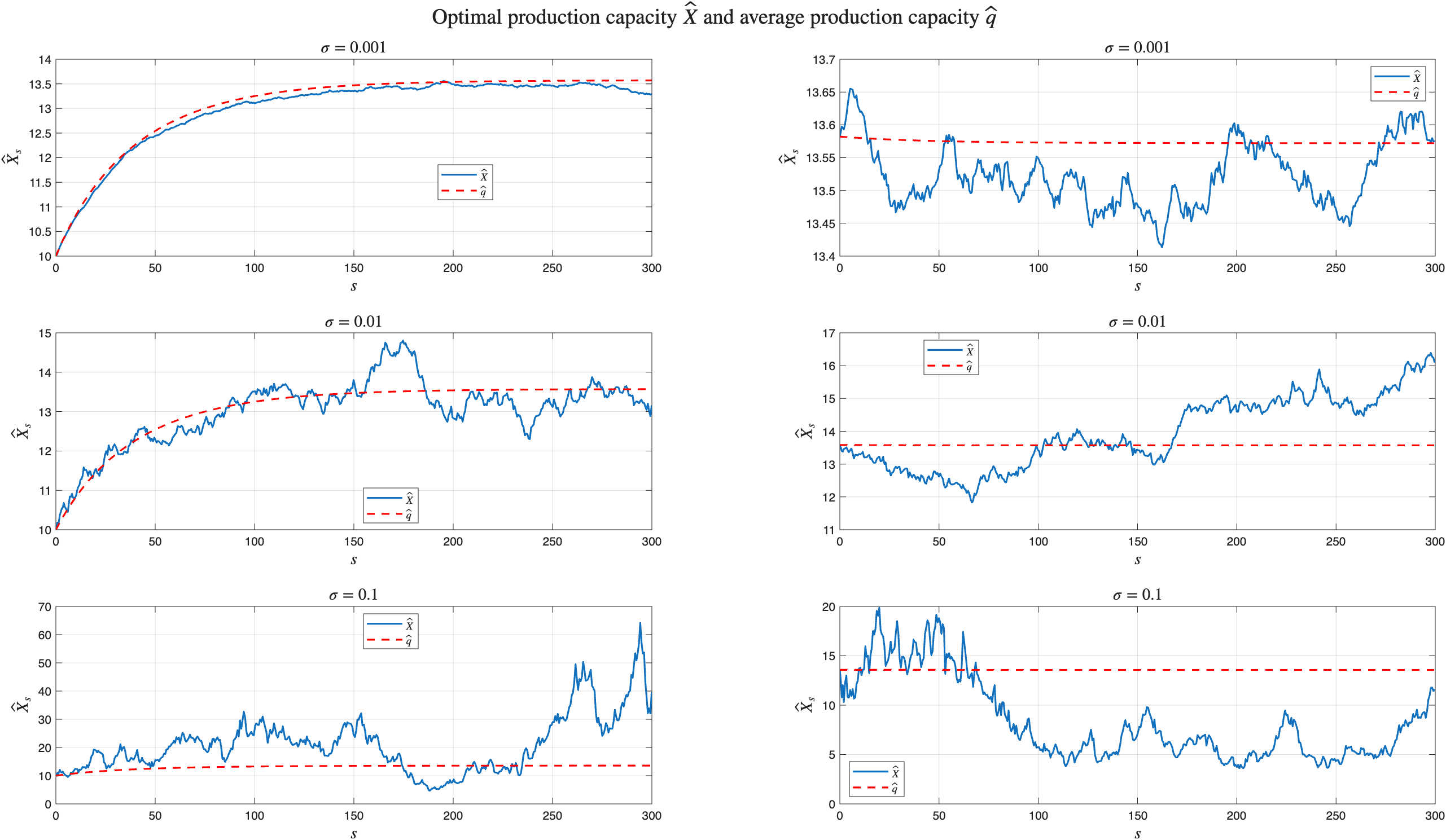}
    \caption{Trajectories of $\widehat X$ and the equilibrium average production capacity $\widehat q$, in the infinite time horizon case, and for $\sigma = 0.001, 0.01, 0.1$. Initial conditions: $x = 10$ in the left column; $x = y_\infty + 0.01 \approx 13.5821$ in the right column.}
    \label{fig:traiett_orizz_infinito}
\end{figure}

Similarly to the finite time horizon case, we observe that $\sigma$ simply accounts for the exposure of the production capacity level to the risk factor modeled by the Brownian motion $B$ in SDE~\eqref{eq:SDE}. Note that in the case $x = y_\infty + 0.01$, $\sigma = 0.001$ (top plot on the right), the level of fluctuation of $\widehat X$ with respect to the average $\widehat q$ is only apparently high, as one can see from the scale of the vertical axis. 
}

\bibliographystyle{plainnat}
\bibliography{Bibliography}
\bigskip
\section*{Statements and Declarations}
\subsection*{Funding}
Giorgio Ferrari gratefully acknowledges financial support from \emph{Deutsche Forschungsgemeinschaft} (DFG, German Research Foundation)– Project-ID 317210226– SFB 1283. This work started during the visit of Alessandro Calvia, Salvatore Federico and Fausto Gozzi at the Center for Mathematical Economics (IMW) at Bielefeld University. These authors thank IMW and the SFB 1283 for the support and hospitality. 

Alessandro Calvia, Salvatore Federico and Fausto Gozzi
are supported by the Italian Ministry of University and Research (MUR),
in the framework of PRIN project 2017FKHBA8 001 (The Time-Space Evolution of Economic Activities: Mathematical
Models and Empirical Applications).

\subsection*{Acknowledgments} The authors wish to thank the anonymous referees and the
associate editor for their kind comments and suggestions that helped to improve the
paper.


\subsection*{Competing Interests}
The authors have no relevant financial or non-financial interests to disclose.
\end{document}